\documentclass[11pt]{article}

\usepackage{amssymb}
\usepackage{amsmath,amsfonts,mathtools}
\usepackage{amsthm}
\usepackage{latexsym}
\usepackage{mathrsfs}
\usepackage{color}
\usepackage{float}
\usepackage{enumitem}
\usepackage{algorithm,algorithmic}
 \allowdisplaybreaks[3]
 
 \usepackage{bbm}
\usepackage{microtype}
\usepackage[a4paper]{geometry}
\usepackage{pifont}
\usepackage{latexsym}

\usepackage[a4paper]{geometry}
\newcommand{\beq}{\begin{equation}}
\newcommand{\eeq}{\end{equation}}
\newcommand{\beqa}{\begin{eqnarray}}
\newcommand{\eeqa}{\end{eqnarray}}
\newcommand{\beqas}{\begin{eqnarray*}}
\newcommand{\eeqas}{\end{eqnarray*}}
\newcommand{\ba}{\begin{array}}
\newcommand{\ea}{\end{array}}
\newcommand{\bi}{\begin{itemize}}
\newcommand{\ei}{\end{itemize}}
\newcommand{\nn}{\nonumber}

\newcommand{\mcX}{{X}}

\newcommand{\mcE}{{\mathcal E}}
\newcommand{\mcN}{{\mathcal N}}

\newcommand{\mcI}{{\mathcal I}}

\newcommand{\dist}{\mathrm{dist}}

\newtheorem{lemma}{Lemma}
\newtheorem{thm}{Theorem}
\newtheorem{coro}{Corollary}

\newtheorem{assumption}{Assumption}

\newtheorem{rem}{Remark}

\newcounter{spb}
\def\bL{{\bar L}}

\def\cC{{\Omega}}

\def\tx{{\tilde x}}

\def\tr0{{\tilde r_0}}

\def\rr{{\mathbb{R}}}

\def\bbE{{\mathbb{E}}}

\def\cB{{\mathscr{B}}}

\def\tK{{\widetilde K}}

\def\mcI{{\mathcal I}}
\def\mcO{{\mathcal O}}
\def\tf{{\tilde f}}
\def\tmcO{{\widetilde{\mathcal{O}}}}
\def\tx{{\tilde x}}

\title{Variance-reduced first-order methods for deterministically constrained stochastic nonconvex optimization with strong convergence guarantees\thanks{This work was partially supported by the Office of Naval Research under Award N00014-24-1-2702,  the Air Force Office of Scientific Research under Award FA9550-24-1-0343, and the National Science Foundation under Awards IIS-2211491 and IIS-2435911. It was primarily conducted during Sanyou Mei’s Ph.D. studies at the University of Minnesota.}}

\author{
Zhaosong Lu
\thanks{
Department of Industrial and Systems Engineering, University of Minnesota, USA (email: {\tt zhaosong@umn.edu}, {\tt xiao0414@umn.edu}). 
}
\and
Sanyou Mei
\thanks{Department of Industrial Engineering and Decision Analytics, the Hong Kong University of Science and Technology, Hong Kong, China (email: {\tt symei@ust.hk}).}
\and
Yifeng Xiao
\footnotemark[2]
}

\date{September 15, 2024 (Revised: October 9, 2024; August 26, 2025)}

\begin{document}
\maketitle

\begin{abstract}
In this paper, we study a class of deterministically constrained stochastic nonconvex optimization problems. Existing methods typically aim to find an $\epsilon$-\emph{expectedly feasible} stochastic stationary point, where the expected violations of both constraints and first-order stationarity are within a prescribed tolerance $\epsilon$. However, in many practical applications, it is crucial that the constraints be nearly satisfied with certainty, making such an $\epsilon$-stochastic stationary point potentially undesirable due to the risk of substantial constraint violations. To address this issue, we propose single-loop variance-reduced stochastic first-order methods, where the stochastic gradient of the stochastic component is computed using either a truncated recursive momentum scheme or a truncated Polyak momentum scheme for variance reduction, while the gradient of the deterministic component is computed exactly. Under the error bound condition with a parameter $\theta \geq 1$ and other suitable assumptions, we establish that these methods respectively achieve sample complexity and first-order operation complexity of $\tilde{\mcO}(\epsilon^{-\max\{\theta+2, 2\theta\}})$ and $\tilde{\mcO}(\epsilon^{-\max\{4, 2\theta\}})$ for finding an $\epsilon$-\emph{surely feasible} stochastic stationary point,\footnote{$\tmcO(\cdot)$  represents $\mathcal{O}(\cdot)$ with logarithmic factors hidden.} where the constraint violation is within $\epsilon$ with \emph{certainty}, and the expected violation of first-order stationarity is within $\epsilon$.  For $\theta=1$, these complexities reduce to $\tmcO(\epsilon^{-3})$ and $\tmcO(\epsilon^{-4})$ respectively, which match, up to a logarithmic factor, the best-known complexities achieved by existing methods for finding an $\epsilon$-stochastic stationary point of unconstrained smooth stochastic nonconvex optimization problems.
\end{abstract}


\medskip


\section{Introduction} \label{intro}
In this paper, we consider constrained stochastic nonconvex optimization problems of the form 
\beq\label{prob}
\begin{aligned}
\min_{x\in\mcX}&\quad f(x) :=\bbE[\tilde f(x,\xi)]\\
\mbox{s.t.}&\quad c(x)=0,
\end{aligned}
\eeq
where $\xi$ is a random variable with sample space $\Xi$, $\tf(\cdot,\xi)$ is continuously differentiable for each $\xi\in\Xi$,  $c:\rr^n\to\rr^m$ is a deterministic smooth mapping, and $\mcX\subseteq \rr^n$ is a simple closed convex set whose projection operator can be evaluated exactly.\footnote{For simplicity, we focus on problem \eqref{prob} with equality constraints only. However, our results can be directly extended to problems involving both equality and inequality constraints; see the concluding remarks for further details.} Problem \eqref{prob} arises in a variety of important areas, including energy systems \cite{shabazbegian2020stochastic}, healthcare \cite{shehadeh2022stochastic}, image processing \cite{luke2020proximal}, machine learning \cite{cuomo2022scientific,karniadakis2021physics},  network optimization \cite{bertsekasnetwork}, optimal control \cite{betts2010practical}, PDE-constrained optimization \cite{rees2010optimal}, resource allocation \cite{fontaine2020adaptive}, and transportation \cite{maggioni2009stochastic}. More applications can be found, for example, in \cite{bottou2018optimization,combettes2011proximal,jiang2018consensus,kushner2012stochastic,liu2015sparse}, and references therein.

Numerous stochastic gradient methods have been developed for solving specific instances of problem \eqref{prob} with $c = 0$ (e.g., see \cite{cutkosky2019momentum,fang2018spider,gao2024non,ghadimi2013stochastic,tran2022hybrid,
wang2019spiderboost,xu2023momentum}). Notably, when $f$ is Lipschitz smooth (see Assumption \ref{a-mm}), the methods in \cite{gao2024non,ghadimi2013stochastic} achieve a sample complexity of $\mcO(\epsilon^{-4})$ for finding an $\epsilon$-stochastic stationary point $x$ that satisfies
\[
\bbE[\dist\left(0,\nabla f(x)+\mcN_\mcX(x)\right)]\leq\epsilon.
\]
Furthermore, when $\tf(\cdot, \xi)$ is Lipschitz smooth on average (see Assumption \ref{a-storm}), the methods in \cite{cutkosky2019momentum,fang2018spider,tran2022hybrid,wang2019spiderboost,xu2023momentum} improve this sample complexity to $\mcO(\epsilon^{-3})$ for finding an $\epsilon$-stochastic stationary point.

Additionally,  various methods have been proposed for problem \eqref{prob} with $X = \rr^n$ and $c \neq 0$. For instance, \cite{kushner1974penalty,wang2017penalty} developed stochastic penalty methods that apply a stochastic approximation or gradient method to solve a sequence of quadratic penalty subproblems.  Stochastic sequential quadratic programming (SQP) methods have also been proposed in \cite{berahas2024stochastic,berahas2021sequential,berahas2023accelerating,curtis2023almost,curtis2024worst,curtis2021inexact,fang2024fully,na2023adaptive,na2022asymptotic,qiu2023sequential}, which modify the classical SQP framework by using stochastic approximations of $f$ and by appropriately selecting step sizes. Under suitable assumptions, these methods ensure the asymptotic convergence of the expected violations of feasibility and first-order stationarity to zero. Moreover, the methods in \cite{curtis2023almost,kushner1974penalty,na2022asymptotic,qiu2023sequential} guarantee almost-sure convergence of these quantities. Besides, the sample complexity of  $\tmcO(\epsilon^{-4})$ for finding an $\epsilon$-stochastic stationary point is achieved by methods in \cite{curtis2024worst,na2022asymptotic}. It is worth mentioning that their operation complexity is often higher than the sample complexity by a constant factor, due to the need to solve linear systems. Additionally, these methods may not be applicable to problem \eqref{prob} when $X \neq \rr^n$.

Recently, several methods have been proposed for solving problem \eqref{prob} with $X \neq \rr^n$ and $c \neq 0$. For instance, \cite{shi2022momentum} proposed a momentum-based linearized augmented Lagrangian method for this problem, achieving a sample and first-order operation complexity\footnote{Sample complexity and first-order operation complexity refer to the total number of samples and gradient evaluations of $\tf$ used throughout the algorithm, respectively.} of  $\tmcO(\epsilon^{-4})$ for finding an $\epsilon$-\emph{expectedly feasible}-stochastic stationary point ($\epsilon$-EFSSP), that is, a point $x$ satisfying 
\beq \label{eps-stationarity1} 
\bbE[\|c(x)\|] \leq \epsilon, \qquad \bbE[\dist\left(0, \nabla f(x) + \nabla c(x)\lambda + \mcN_\mcX(x)\right)] \leq \epsilon 
\eeq 
for some Lagrangian multiplier $\lambda$. This complexity improves to  $\tmcO(\epsilon^{-3})$ when a nearly feasible point of \eqref{prob} is available.  More recently, \cite{alacaoglu2024complexity} proposed a stochastic quadratic penalty method that iteratively applies a single stochastic gradient descent step to a sequence of quadratic penalty functions $Q_{\rho_k}(x)$, where $\rho_k$ is a penalty parameter, and $Q_{\rho}$ is defined as 
\beq\label{def-Q} 
Q_\rho(x) := f(x) + \frac{\rho}{2} \|c(x)\|^2. 
\eeq 
In this method, the stochastic gradient is computed using the recursive momentum scheme introduced in  \cite{cutkosky2019momentum}, treating $\rho$ as part of the variables (see Section \ref{sec:storm} for more detailed discussions). Under the error bound condition \eqref{err-cond} with $\theta=1$ and other suitable assumptions, this method achieves a sample and first-order operation complexity of $\tmcO(\epsilon^{-4})$ for finding an $\epsilon$-EFSSP.

In many applications such as energy systems \cite{shabazbegian2020stochastic}, machine learning \cite{cuomo2022scientific,karniadakis2021physics}, resource allocation \cite{fontaine2020adaptive}, and transportation \cite{maggioni2009stochastic}, all or some of the constraints in problem \eqref{prob} are hard constraints representing imperative requirements. Consequently, any desirable approximate solution must (nearly) satisfy these constraints. As mentioned above, the $\epsilon$-stochastic stationary point $x$ found by existing methods \cite{alacaoglu2024complexity,berahas2021sequential,curtis2024worst,curtis2021inexact,shi2022momentum,wang2017penalty}   
 satisfies $\bbE[\|c(x)\|]\leq\epsilon$. However, it is possible that $\|c(x)\|$ may still be excessively large, leading to significant constraint violations, which is undesirable in applications where practitioners require nearly exact constraint satisfaction.

To address the aforementioned issue, we propose single-loop variance-reduced stochastic first-order methods for solving problem \eqref{prob}, inspired by the framework of \cite[Algorithm 2]{alacaoglu2024complexity}, but with a significantly different approach to constructing the stochastic gradient. Specifically, starting from any initial point $x_0\in X$, we iteratively solve a sequence of quadratic penalty problems $\min_{x\in X} Q_{\rho_k} (x)$ by performing only a \emph{single} stochastic gradient descent step
\[
x_{k+1}=\Pi_\mcX(x_k-\eta_kG_k),
\]
 where  $\eta_k>0$ is a step size,  $G_k$ is a variance-reduced estimator of $\nabla Q_{\rho_k} (x_k)$, and $\Pi_\mcX$ denotes the projection operator onto the set $X$. In our approach,  $G_k$ is constructed by  separately handling the stochastic component $f(x)$ and the deterministic penalty term $\rho_k\|c(x)\|^2/2$ of $Q_{\rho_k}(x)$.  The gradient $\nabla f(x_k)$ is approximated by a stochastic estimator $g_k$, computed via a novel truncated recursive or Polyak momentum scheme, which ensures both variance reduction and boundedness of $g_k$---a property not guaranteed by the standard recursive momentum scheme  \cite{cutkosky2019momentum}. Meanwhile, the gradient of the penalty term is computed exactly as $\rho_k \nabla c(x_k) c(x_k)$. Combining these two components yields $G_k=g_k+\rho_k\nabla c(x_k)c(x_k)$, which serves as a variance-reduced stochastic estimator of $\nabla Q_{\rho_k}(x_k)$ (see Algorithms \ref{thm:storm} and \ref{thm:mm} for futher details). 

Moreover,  we develop a novel convergence analysis for our proposed methods, which differs significantly from existing analyses of stochastic algorithms for solving problem \eqref{prob}. Specifically, by leveraging the variance-reduced structure of $G_k$, we first interpret our algorithms as inexact projected gradient methods applied to the associated feasibility problem and establish convergence rates for the \emph{deterministic} feasibility violation.  We then use this result, along with a carefully constructed potential function, to derive convergence rates for the expected first-order stationarity violation.

Under  the error bound condition \eqref{err-cond} with $\theta \geq 1$ and other suitable assumptions, we establish that our methods respectively achieve a sample and first-order operation complexity of $\tmcO\big(\epsilon^{-\max\{\theta+2,2\theta\}}\big)$ and $\tmcO\big(\epsilon^{-\max\{4,2\theta\}}\big)$ for finding an $\epsilon$-\emph{surely feasible}-stochastic stationary point ($\epsilon$-SFSSP), that is, a random vector $x$ satisfying 
\beq \label{eps-stationarity2}
\|c(x)\| \leq\epsilon, \qquad  \bbE[\dist\left(0,\nabla f(x)+\nabla c(x)\lambda+\mcN_\mcX(x)\right)]\leq\epsilon
\eeq
for some random vector $\lambda$, where the randomness of $(x,\lambda)$ arises from the stochastic samples used in the methods. Since an $\epsilon$-SFSSP nearly satisfies all the constraints with certainty, it is a stronger notion than the $\epsilon$-EFSSP commonly considered in the literature.  Moreover, our complexity results are novel, as no previous work has established complexity bounds for general values of $\theta$, even for finding the weaker $\epsilon$-EFSSP. Additionally, in the special case where $\theta = 1$, our method based on a truncated recursive momentum scheme achieves a sample and first-order operation complexity of $\tmcO(\epsilon^{-3})$, which matches the complexity achieved by the method in \cite{shi2022momentum}, but is stronger than the complexity of $\tmcO(\epsilon^{-4})$ achieved by \cite[Algorithm 2]{alacaoglu2024complexity}. However, these methods achieve such complexities only for finding the weaker $\epsilon$-EFSSP, and the method in \cite{shi2022momentum} additionally requires the availability of a nearly feasible point. Moreover, the complexity of $\tmcO(\epsilon^{-3})$ matches, up to a logarithmic factor, the complexity achieved by the method in \cite{cutkosky2019momentum} for problem \eqref{prob} with $X=\rr^n$ and $c=0$, as well as the methods in \cite{fang2018spider, gao2024non, ghadimi2013stochastic, tran2022hybrid, wang2019spiderboost, xu2023momentum} for problem \eqref{prob} with $c=0$.

The main contributions of our paper are summarized as follows.
\begin{itemize}
\item We propose novel single-loop variance-reduced stochastic first-order methods with a truncated recursive or Polyak momentum for solving problem \eqref{prob}.
\item   We show that under the error bound condition \eqref{err-cond} with $\theta \geq 1$ and other suitable assumptions, our proposed methods respectively achieve a sample and first-order operation complexity of $\tmcO\big(\epsilon^{-\max\{\theta+2,2\theta\}}\big)$ and $\tmcO\big(\epsilon^{-\max\{4,2\theta\}}\big)$ for finding an $\epsilon$-SFSSP, which is a stronger notion than the $\epsilon$-EFSSP commonly sought by existing methods.
\end{itemize}
To the best of our knowledge, this is the first work to develop methods with provable complexity guarantees for finding an approximate stochastic stationary point of problem \eqref{prob} that nearly satisfies all constraints with \emph{certainty}. 

The rest of this paper is organized as follows. In Subsection \ref{sec:notation}, we introduce some notation,  terminology, and assumption. In Sections \ref{sec:storm} and  \ref{sec:mm}, we propose stochastic first-order methods with a truncated recursive momentum or a truncated Polyak momentum for problem \eqref{prob} and analyze their convergence. We provide the proof of the main results in Section \ref{sec:proof}. Finally, concluding remarks are given in Section \ref{sec:conclude}.

\subsection{Notation, terminology, and assumption}\label{sec:notation}
The following notation will be used throughout this paper.  Let $\rr_{>0}$ denote the set of positive real numbers, and  $\rr^n$ denote the Euclidean space of dimension $n$. The standard inner product and Euclidean norm are denoted by $\langle\cdot,\cdot\rangle$ and $\|\cdot\|$, respectively.  For any $r>0$, let $\cB(r)$ represent the Euclidean ball centered at the origin with radius $r$, that is, $\cB(r)=\{x:\|x\|\leq r\}$.  For any $t\in\rr$, let $t_+=\max\{t,0\}$ and $\lceil t \rceil$ denote the least integer greater than or equal to $t$.

A mapping $\phi$ is said to be \emph{$L_{\phi}$-Lipschitz continuous} on a set $\Omega$ if $\|\phi(x)-\phi(x')\| \leq L_{\phi} \|x-x'\|$ for all $x,x'\in \Omega$. Also, it is said to be \emph{$L_{\nabla\phi}$-smooth} on $\Omega$ if $\|\nabla\phi(x)-\nabla\phi(x')\| \leq L_{\nabla\phi} \|x-x'\|$ for all $x,x'\in \Omega$, where $\nabla \phi$ denotes the transpose of the Jacobian of $\phi$.  Given a nonempty closed convex set $\cC$, $\dist(x,\cC)$ denotes the Euclidean distance from $x$ to $\cC$, and $\Pi_\cC(x)$ denotes the Euclidean projection of $x$ onto $\Omega$.  In addition,  the normal cone of $\cC$ at any $x\in \cC$ is denoted by $\mcN_\cC(x)$.
Finally, we use $\tmcO(\cdot)$ to denote the asymptotic upper bound that ignores logarithmic factors.

Throughout this paper, we make the following assumptions for problem \eqref{prob}.

\begin{assumption}\label{a1}
\begin{enumerate}[label=(\roman*)]
\item The optimal value $f^*$ of problem \eqref{prob}  and $Q_1^*:= \min_{x\in X} \{f(x)+\|c(x)\|^2/2\}$ are finite.
\item $f$ is differentiable and $L_f$-Lipschitz continuous in an open neighborhood of $\mcX$.
\item For each $\xi\in\Xi$, $\tf(\cdot,\xi)$ is differentiable on $\mcX$ and satisfies the following conditions:
\[
\bbE[\nabla\tf(x,\xi)]=\nabla f(x),\quad\bbE[\|\nabla\tf(x,\xi)-\nabla f(x)\|^2]\leq\sigma^2 \qquad \forall x\in\mcX
\]
for some constant $\sigma \geq 0$. 
\item The mapping $c$ is $L_c$-Lipschitz continuous and $L_{\nabla c}$-smooth in an open neighborhood of $\mcX$. Additionally, $\|c(x)\|\leq C_c$ for all $x\in\mcX$, and there exist constants $\gamma>0$ and $\theta\geq1$ such that
\beq \label{err-cond}
\dist\left(0,\nabla c(x)c(x)+\mcN_\mcX(x)\right)\geq\gamma\|c(x)\|^\theta\qquad \forall x\in\mcX.
\eeq
\end{enumerate}
\end{assumption}

In addition, for notational convenience, we define
\begin{align}
L:= L_c^2+C_cL_{\nabla c}.\label{def-L}
\end{align}
It follows from this and Assumption \ref{a1} that $\|c(x)\|^2/2$ is $L$-smooth on $\mcX$, and 
\beq\label{cf-bnd}
\|\nabla f(x)\|\leq L_f,\quad\|\nabla c(x)\|\leq L_c\qquad\forall x\in\mcX.
\eeq

Before ending this subsection, we make some remarks on Assumption \ref{a1}.

\begin{rem} 
\begin{enumerate}[label=(\roman*)]
\item The assumption on the finiteness of $Q_1^*$ is generally weaker than the condition $\min_{x \in X} f(x) \geq 0$, which is imposed in related work such as \cite{alacaoglu2024complexity}. Moreover, this assumption is quite mild. Specifically, since the optimal value $f^*$ of \eqref{prob} is finite and 
\[
\lim_{\rho \to \infty} \min_{x \in X} \{f(x) + \rho \|c(x)\|^2/2\} = f^*,
\]
 there exists some $\underline \rho>0$ such that $\min_{x\in X}\{f(x)+\rho\|c(x)\|^2/2\}$ and consequently $\min_{x\in X}\{\rho^{-1}f(x)+\|c(x)\|^2/2\}$ are finite for all $\rho \geq \underline \rho$. Therefore, if Assumption \ref{a1}(i) does not hold, one can replace $f$ with $\rho^{-1}f$ for some $\rho \geq \underline \rho$, ensuring the resulting problem \eqref{prob} satisfies Assumption \ref{a1}(i). 

\item Assumption \ref{a1}(iii) is standard and implies that $\nabla\tf(x,\xi)$ is an unbiased estimator of $\nabla f(x)$ with a bounded variance for all $x\in X$. 

\item Condition \eqref{err-cond} is an error bound condition that plays a crucial role in designing algorithms capable of producing nearly feasible solutions to problem \eqref{prob}. The special case with $\theta = 1$ is commonly assumed in the literature (e.g., see \cite{alacaoglu2024complexity,li2024stochastic,lin2022complexity,sahin2019inexact}). When $\theta \in [1, 2)$, one can verify that condition \eqref{err-cond} is equivalent to a global Kurdyka–\L ojasiewicz (KL) condition 
\cite{kurdyka1998gradients,lojasiewicz1963propriete} with exponent $1 - \theta/2$ for the feasibility problem $\min_x \{\|c(x)\|^2/2 + \delta_X(x)\}$,  where $\delta_X(\cdot)$ denotes the indicator function of the set $X$. Moreover, when $\theta \geq 2$, condition \eqref{err-cond} goes beyond the KL condition and is thus more general.
\end{enumerate}
\end{rem}

\section{A stochastic first-order method with a truncated  recursive momentum for problem \eqref{prob}}\label{sec:storm}

In this section, we propose a stochastic first-order method with a truncated recursive momentum for solving problem \eqref{prob}, inspired by the framework of \cite[Algorithm 2]{alacaoglu2024complexity}, but employing a significantly different approach to constructing the stochastic gradient.  Moreover, the proposed method exhibits stronger convergence properties compared to existing methods (see Remark \ref{remark2}).

Specifically,  starting from any initial point $x_0\in X$, we approximately solve a sequence of quadratic penalty problems $\min_{x\in X} Q_{\rho_k} (x)$ by performing only a \emph{single} stochastic gradient descent step $x_{k+1}=\Pi_\mcX(x_k-\eta_kG_k)$, where $\rho_k$ is a penalty parameter, $\eta_k>0$ is a step size,  $G_k$ is a variance-reduced estimator of $\nabla Q_{\rho_k} (x_k)$, and $Q_{\rho_k}$ is given in \eqref{def-Q}. 
Notice from \eqref{def-Q} that $\nabla Q_{\rho_k}(x_k)=\nabla f(x_k)+\rho_k\nabla c(x_k) c(x_k)$. Based on this, we particularly choose $G_k=g_k+\rho_k\nabla c(x_k)c(x_k)$, where $g_k$ is a variance-reduced estimator of $\nabla f(x_k)$, computed recursively as follows:
\beq \label{gk}
g_k=\Pi_{\cB(L_f)}\big(\nabla \tf(x_k,\xi_k)+(1-\alpha_{k-1})(g_{k-1}-\nabla\tf(x_{k-1},\xi_k))\big)
\eeq
for some $\alpha_{k-1}\in(0,1]$ and a randomly drawn sample $\xi_k$. 
This scheme is a slight modification of the recursive momentum scheme introduced in \cite{cutkosky2019momentum}, incorporating a truncation operation via the projection operator $\Pi_{\cB(L_f)}$ to ensure the boundedness of $\{g_k\}$, which is crucial for the subsequent analysis. Interestingly, despite this truncation, the modified scheme preserves a variance-reduction property similar to the original scheme in \cite{cutkosky2019momentum} (see Lemma \ref{l-exp}).

The proposed stochastic first-order method with a truncated recursive momentum for solving problem \eqref{prob} is presented in Algorithm \ref{alg1} below.

\begin{algorithm}[H]
\caption{A stochastic first-order method with a truncated  recursive momentum for problem \eqref{prob}}\label{alg1}
\begin{algorithmic}[1]
\REQUIRE $x_1\in\mcX$, $\{\alpha_k\} \subset (0,1]$,  $\{\rho_k\}, \{\eta_k\} \subset \rr_{>0}$, and $L_f$ given in Assumption \ref{a1}.
\STATE Sample $\xi_1$ and set $g_1=\Pi_{\cB(L_f)}(\nabla\tf(x_1,\xi_1))$.
\FOR{$k=1,2,\dots$}
\STATE $G_k=g_k+\rho_k\nabla c(x_k)c(x_k)$.
\STATE $x_{k+1}=\Pi_\mcX(x_k-\eta_kG_k)$.
\STATE Sample $\xi_{k+1}$ and set $g_{k+1}=\Pi_{\cB(L_f)}\big(\nabla\tf(x_{k+1},\xi_{k+1})+(1-\alpha_k)(g_k-\nabla\tf(x_k,\xi_{k+1}))\big)$.
\ENDFOR
\end{algorithmic}
\end{algorithm}

The parameters $\{\alpha_k\}$, $\{\rho_k\}$ and $\{\eta_k\}$ will be specified in Theorem \ref{thm:storm} for Algorithm \ref{alg1} to achieve a desirable convergence rate. While Algorithm \ref{alg1} shares a similar framework with \cite[Algorithm 2]{alacaoglu2024complexity}, the construction of the variance-reduced estimator $G_k$ of $\nabla Q_{\rho_k} (x_k)$ differs substantially between the two algorithms. In particular, $G_k$ in \cite[Algorithm 2]{alacaoglu2024complexity} is derived by applying the recursive momentum scheme introduced in \cite{cutkosky2019momentum} to the entire function $Q_{\rho}(x)$, treating $\rho$ as part of the variables, and it is given by 
\[
G_k=\widetilde\nabla Q_{\rho_k}(x_k,\xi_k)+(1-\alpha_{k-1})(G_{k-1}-\widetilde\nabla Q_{\rho_{k-1}}(x_{k-1},\xi_k)),
\] 
where $\widetilde\nabla Q_{\rho}(x,\xi)=\nabla \tf(x,\xi)+\rho\nabla c(x)c(x)$. In contrast, $G_k$  in Algorithm \ref{alg1} is constructed by handling the stochastic part 
$f(x)$ and the deterministic part $\rho_k\|c(x)\|^2/2$ of $Q_{\rho_k}(x)$  separately. Specifically, $\nabla f(x_k)$ is approximated by a stochastic estimator $g_k$ computed via  truncated  recursive momentum scheme as given in \eqref{gk}, while the gradient of the deterministic term, $\nabla (\rho_k\|c(x)\|^2/2)|_{x=x_k}$, is evaluated exactly as $\rho_k\nabla c(x_k)c(x_k)$. Combining these two components gives $G_k=g_k+\rho_k\nabla c(x_k)c(x_k)$ for Algorithm \ref{alg1}. 

Moreover, our convergence analysis for Algorithm \ref{alg1} substantially differs from that of \cite[Algorithm 2]{alacaoglu2024complexity}. Specifically, by leveraging the structure of $G_k$, we first interpret our algorithm as an inexact projected gradient method applied to the associated feasibility problem, and establish a convergence rate for the \emph{deterministic} feasibility violation. We then use this result, together with a carefully constructed potential function, to derive a convergence rate for the expected first-order stationarity violation. In contrast, \cite{alacaoglu2024complexity} first bounds the expected feasibility violation using the expected consecutive change of a potential function and uses this bound to derive a convergence rate for the expected first-order stationarity violation. This rate is then used to establish a convergence rate for the expected feasibility violation.

Our novel algorithmic design and analysis lead to significantly stronger convergence results for Algorithm \ref{alg1} than those established for \cite[Algorithm 2]{alacaoglu2024complexity}. Specifically, under Assumptions \ref{a1} and \ref{a-storm} with $\theta=1$,  Algorithm \ref{alg1} generates a  sequence $\{x_k\}$ satisfying 
\[
 \|c(x_{\iota_k})\|^2= \widetilde\mcO(k^{-2/3}), \quad \bbE\left[\dist^2\left(0,\nabla f(x_{\iota_k})+\nabla c(x_{\iota_k})\lambda_{\iota_k}+\mcN_\mcX(x_{\iota_k})\right)\right] = \widetilde\mcO(k^{-2/3})
\]
for some sequence $\{\lambda_k\}$, where $\iota_k$  is uniformly drawn from $\left\{\lceil k/2\rceil+1,\dots,k\right\}$ for $k\geq 2$ (see Theorem \ref{thm:storm}). In contrast, \cite[Algorithm 2]{alacaoglu2024complexity}  generates a sequence $\{\tx_k\}$ satisfying 
\[
\bbE[\|c(\tx_{\iota_k})\|^2]= \widetilde\mcO(k^{-1/2}), \quad \bbE\left[\dist^2\left(0,\nabla f(\tx_{\iota_k})+\nabla c(\tx_{\iota_k})\tilde\lambda_{\iota_k}+\mcN_\mcX(\tx_{\iota_k})\right)\right] = \widetilde\mcO(k^{-1/2})
\]
for some sequence $\{\tilde\lambda_k\}$ (see \cite[Theorem 4.2]{alacaoglu2024complexity}). Clearly, the sequence $\{x_k\}$ generated by Algorithm \ref{alg1} exhibits a stronger convergence property: it not only achieves a substantially faster convergence rate, but also ensures that $\|c(x_{\iota_k})\|^2$ converges \emph{deterministically}, rather than merely in expectation. Furthermore, under Assumptions \ref{a1} and \ref{a-storm} with $\theta > 1$,  Algorithm \ref{alg1} guarantees that
\[
 \|c(x_{\iota_k})\|^2= \widetilde\mcO(k^{-\tau}), \quad \bbE\left[\dist^2\left(0,\nabla f(x_{\iota_k})+\nabla c(x_{\iota_k})\lambda_{\iota_k}+\mcN_\mcX(x_{\iota_k})\right)\right] = \widetilde\mcO(k^{-\tau})
\]
with $\tau=\min\{2/(\theta+2),1/\theta\}$ for some sequence $\{\lambda_k\}$, while the convergence behavior of \cite[Algorithm 2]{alacaoglu2024complexity} under this setting remains unknown.

To formally present the convergence results for Algorithm \ref{alg1}, we next introduce an assumption that specifies an average smoothness condition for problem \eqref{prob}.

\begin{assumption}\label{a-storm}
The function $\tf(x,\xi)$ satisfies the average smoothness condition:
\[
\bbE[ \|\nabla\tf(u,\xi)-\nabla\tf(v,\xi)\|^2]\leq \bL_{\nabla f}^2\|u-v\|^2\qquad \forall u,v\in\mcX.
\]
\end{assumption}

Assumption \ref{a-storm} is commonly imposed in the literature to design algorithms for solving problems of the form $\min_x \bbE[\tf(x,\xi)]+P(x)$, where $P$ is either zero or a simple but possibly nonsmooth function (e.g., see \cite{cutkosky2019momentum,fang2018spider,tran2022hybrid,wang2019spiderboost,xu2023momentum}). 
It can be observed that Assumption \ref{a-storm} implies that $\nabla f$ is $\bL_{\nabla f}$-smooth on $X$, that is,
\beq \label{bL-smooth}
 \|\nabla f(u)-\nabla f(v)\| \leq \bL_{\nabla f} \|u-v\| \qquad \forall u,v\in\mcX.
\eeq
However, the reverse implication does not hold in general (e.g., see \cite{ghadimi2013stochastic}).

We are now ready to present the convergence results for Algorithm \ref{alg1}, with the proof deferred to Subsection \ref{sec:proof-storm}. Specifically, we will present convergence rates for the following two 
quantities:
\beq \label{stationarity-criteria}
 \|c(x_{\iota_k})\|^2 \ \   \mbox{and} \ \   \bbE\left[\dist^2\left(0,\nabla f(x_{\iota_k})+\rho_{\iota_k-1}\nabla c(x_{\iota_k})c(x_{\iota_k})+\mcN_\mcX(x_{\iota_k})\right)\right],
\eeq
 where $\iota_k$ is uniformly drawn from $\left\{\lceil k/2\rceil+1,\dots,k\right\}$.  These quantities measure the constraint violation and the expected stationarity violation at $x_{\iota_k}$.

\begin{thm}\label{thm:storm}
Suppose that Assumptions \ref{a1} and \ref{a-storm} hold, and $\{x_k\}$ is generated by Algorithm \ref{alg1}. Let $L$ be defined in \eqref{def-L}, $L_f$, $\bL_{\nabla f}$, $L_c$, $C_c$, $\sigma$,  $\gamma$, $\theta$ and $Q_1^*$ be given in Assumptions \ref{a1} and \ref{a-storm}, $g_1$ be given in Algorithm \ref{alg1}, and $\iota_k$  be the random variable uniformly generated from $\left\{\lceil k/2\rceil+1,\dots,k\right\}$ for $k\geq 2$, and let $\rho_k$, $\eta_k$ and $\alpha_k$ be chosen as
\beq\label{def-para2}
\rho_k=k^\nu,\quad\eta_k=k^{-\nu}/(4\log(k+2)), \quad \alpha_k=k^{-2\nu},\quad\mbox{where}\quad \nu=\min\{\hat\theta/(\hat\theta+2),1/2\}
\eeq
for some $\hat\theta\geq1$. Then for all $k\geq2\tK_1$, we have
\begin{align*}
&\bbE\left[\dist^2\left(0,\nabla f(x_{\iota_k})+\rho_{\iota_k-1}\nabla c(x_{\iota_k})c(x_{\iota_k})+\mcN_\mcX(x_{\iota_k})\right)\right]\\
&\leq\frac{102\log(k+1)}{(k-1)^{1-\nu}}\Bigg(Q_1(x_1)-Q_1^*+\frac{\|g_1-\nabla f(x_1)\|^2}{2\log 3}+\frac{1}{2}C_1\max\{1,k^{\nu-\frac{2\nu}{\theta}}\}(1+\log k)\nn \\ 
&\quad+\frac{1}{2}C_c^2\big(\tK_1^{\nu}-1\big)
+\frac{3\sqrt{2}\sigma^2(1+\log k)}{2\log3}+\frac{1}{16}(1+\log\tK_1)\big(\bL_{\nabla f}+\tK_1^{\frac{1}{2}}L+6\tK_1^{\frac{1}{2}}\bL_{\nabla f}^2\big)(L_f^2+C_c^2L_c^2\tK_1)\Bigg),\\
&\|c(x_{\iota_k})\|^2\leq2C_1 (k/2)^{-2\nu/\theta},
\end{align*}
where
\begin{align}
&\tK_1=\left\lceil\max\left\{(2\bL_{\nabla f})^{\frac{1}{\nu}},e^{6\times2^{\nu/2}\bL_{\nabla f}^2},e^{2L},e^{2\theta},\left(e^{-1}\gamma^{-2}2^{6-\theta}\log(e^{2\theta}+2)\right)^{2\theta}\right\}\right\rceil,\label{def-tK2}\\
&C_1=\max\left\{1,\tK_1^{2\nu/\theta}C_c^2/2,2^{3-\theta}L_f^2\gamma^{-2}\right\}.\label{def-nuC2}
\end{align}

\end{thm}

\begin{rem} \label{remark2} 
\begin{enumerate}[label=(\roman*)]
\item The parameter $\hat\theta$ in Theorem \ref{thm:storm} represents an estimate of the actual value of $\theta$. If the actual value of $\theta \geq 1$ is known, we set $\hat\theta = \theta$.
\item As shown in Theorem \ref{thm:storm}, the choice of $\rho_k$, $\eta_k$, and $\alpha_k$ in \eqref{def-para2} with any arbitrary $\hat\theta \geq 1$ ensures a convergence rate of $\widetilde \mcO(k^{-\min\{2\nu/\theta, 1-\nu, 1-2\nu+2\nu/\theta\}})$ for the quantities in \eqref{stationarity-criteria}. When the actual value of $\theta \geq 1$ is known, we set $\hat\theta = \theta$, and the convergence rate improves to $\widetilde \mcO(k^{-\min\{2/(\theta+2),\theta^{-1}\}})$.
\end{enumerate}
\end{rem}

The following result is an immediate consequence of Theorem \ref{thm:storm}. It provides iteration complexity results for Algorithm \ref{alg1} to find an $\epsilon$-SFSSP $x_{\iota_k}$ of problem \eqref{prob} satisfying \eqref{kkt} below. 

\begin{coro} \label{coro:storm}
Suppose that Assumptions \ref{a1} and \ref{a-storm} hold, and $\{x_k\}$ is generated by Algorithm \ref{alg1}. Let $\theta \geq 1$ be given in Assumption \ref{a1}, and $\iota_k$  be the random variable uniformly generated from $\left\{\lceil k/2\rceil+1,\dots,k\right\}$ for $k\geq 2$. Then the following statements hold.
\begin{enumerate}[label=(\roman*)]
\item Suppose that the actual value of $\theta$ is known. Let $\rho_k$, $\eta_k$ and $\alpha_k$ be chosen as in \eqref{def-para2} with $\hat\theta=\theta$. Then for any $\epsilon>0$, there exists some $T=\tmcO(\epsilon^{-\max\{\theta+2,2\theta\}})$ such that 
\beq\label{kkt}
\|c(x_{\iota_k})\|\leq\epsilon, \quad  \bbE\left[\dist\left(0,\nabla f(x_{\iota_k})+\rho_{\iota_k-1}\nabla c(x_{\iota_k})c(x_{\iota_k})+\mcN_\mcX(x_{\iota_k})\right)\right]\leq\epsilon 
\eeq
hold for all $k\geq T$.
\item Suppose that the actual value of $\theta$ is unknown. Let $\rho_k$, $\eta_k$ and $\alpha_k$ be chosen as in \eqref{def-para2} with $\hat\theta=1$. Then for any $\epsilon>0$,  there exists some $ T=\tmcO\left(\epsilon^{-3\theta}\right)$ such that \eqref{kkt} holds for all $k\geq T$.
\item Suppose that the actual value of $\theta$ is unknown. Let $\rho_k$, $\eta_k$ and $\alpha_k$ be chosen as in \eqref{def-para2} with $\hat\theta=2$. Then for any $\epsilon>0$,  there exists some $ T=\tmcO\left(\epsilon^{-\max\{4,2\theta\}}\right)$ such that \eqref{kkt} holds for all $k\geq T$.
\end{enumerate}
\end{coro}

\begin{rem}
\begin{enumerate}[label=(\roman*)]
\item Since Algorithm \ref{alg1} requires one sample, one gradient evaluation of $c$, and two gradient evaluations of $\tf$ per iteration, its sample and first-order operation complexity are of the same order as its iteration complexity.  Our complexity results stated in Corollary \ref{coro:storm} are novel, as no prior work has established complexity bounds for general values of $\theta$, even for finding the weaker $\epsilon$-EFSSP satisfying \eqref{eps-stationarity1}. 
\item For $\theta = 1$, Algorithm \ref{alg1} with $\rho_k$, $\eta_k$ and $\alpha_k$ chosen as in \eqref{def-para2} with $\hat\theta=1$ achieves a sample and first-order operation complexity of $\tmcO(\epsilon^{-3})$, which matches the complexity achieved by the method in \cite{shi2022momentum}, but is stronger than the complexity of $\tmcO(\epsilon^{-4})$ achieved by \cite[Algorithm 2]{alacaoglu2024complexity}. However, these methods achieve such complexities only for finding the weaker $\epsilon$-EFSSP satisfying \eqref{eps-stationarity1}, and the method in \cite{shi2022momentum} additionally requires the availability of a nearly feasible point. Moreover, the complexity of $\tmcO(\epsilon^{-3})$ matches, up to a logarithmic factor, the complexity achieved by the method in \cite{cutkosky2019momentum} for problem \eqref{prob} with $X=\rr^n$ and $c=0$, as well as the methods in \cite{fang2018spider, gao2024non, ghadimi2013stochastic, tran2022hybrid, wang2019spiderboost, xu2023momentum} for problem \eqref{prob} with $c=0$.
\item When $\theta \in [1,4/3)$ and its actual value is unknown,  Algorithm \ref{alg1} achieves better complexity by setting $\rho_k$, $\eta_k$ and $\alpha_k$ as in \eqref{def-para2} with $\hat\theta=1$ rather than $\hat\theta=2$. However, when $\theta >4/3$ and its actual value is unknown, a better complexity is achieved with $\hat\theta=2$ instead of $\hat\theta=1$.
\end{enumerate}
\end{rem}

\section{A stochastic first-order method with a truncated Polyak momentum for problem \eqref{prob}}\label{sec:mm}

In this section, we propose a stochastic first-order method with a truncated Polyak momentum for solving problem \eqref{prob}. This method modifies Algorithm \ref{alg1}, with $g_k$ being recursively generated using the following truncated Polyak momentum scheme: 
\[
g_k=\Pi_{\cB(L_f)}\big(\alpha_{k-1}\nabla \tf(x_k,\xi_k)+(1-\alpha_{k-1})g_{k-1}\big)
\]
 for some $\alpha_{k-1}\in(0,1]$ and a randomly drawn sample $\xi_k$, where $L_f$ is given in Assumption \ref{a1}. This scheme is a slight modification of the well-known Polyak momentum scheme \cite{gao2024non,polyak1964some,yu2019linear}, incorporating a truncation operation via the projection operator $\Pi_{\cB(L_f)}$ to ensure the boundedness of the sequence $\{g_k\}$.  This boundedness is crucial for our subsequent analysis. Despite the truncation, the modified scheme preserves the variance-reduction property of the original Polyak momentum scheme (see Lemma \ref{l-exp2}).

 The proposed stochastic first-order method with a truncated Polyak momentum is presented in Algorithm \ref{alg2}.

\begin{algorithm}[H]
\caption{A stochastic first-order method with a truncated Polyak momentum for  \eqref{prob}}\label{alg2}
\begin{algorithmic}[1]
\REQUIRE $x_1\in\mcX$, $\{\alpha_k\} \subset (0,1]$,  and $\{\rho_k\}, \{\eta_k\} \subset \rr_{>0}$, and $L_f$ given in Assumption \ref{a1}.
\STATE Sample $\xi_1$ and set $g_1=\Pi_{\cB(L_f)}\big(\nabla \tf(x_1,\xi_1)\big)$.
\FOR{$k=1,2,\dots$}
\STATE $G_k=g_k+\rho_k\nabla c(x_k)c(x_k)$.
\STATE $x_{k+1}=\Pi_\mcX(x_k-\eta_kG_k)$.
\STATE Sample $\xi_{k+1}$ and set $g_{k+1}=\Pi_{\cB(L_f)}\big((1-\alpha_k)g_k+\alpha_k\nabla \tf(x_{k+1},\xi_{k+1})\big)$.
\ENDFOR
\end{algorithmic}
\end{algorithm}

It is worth noting that if the set $\{\nabla \tf(x,\xi): x\in X, \xi\in\Xi\}$ is bounded, the recursion for $g_{k+1}$ in step 5 of Algorithm \ref{alg2} can be modified to $g_{k+1}=(1-\alpha_k)g_k+\alpha_k\nabla \tf(x_{k+1},\xi_{k+1})$. This modification ensures the boundedness of $\{g_k\}$, and the resulting algorithm still retains the same rate of convergence as Algorithm \ref{alg2}.

Similar to Algorithm \ref{alg1}, we provide a novel convergence analysis for Algorithm \ref{alg2}. Specifically, by leveraging the structure of $G_k$, we first interpret Algorithm \ref{alg2} as an inexact projected gradient method applied to the associated feasibility problem and establish a convergence rate for the \emph{deterministic} feasibility violation. We then use this result, together with a carefully constructed potential function, to derive a convergence rate for the expected first-order stationarity violation. This new analysis framework allows us to establish a \emph{deterministic} convergence guarantee on the constraint violation.

To present the convergence results for Algorithm \ref{alg2}, we introduce the following assumption, which specifies a Lipschitz smoothness condition for problem \eqref{prob}.

\begin{assumption}\label{a-mm}
The function $f$  is $L_{\nabla f}$-smooth on $\mcX$, that is, 
\[
 \|\nabla f(u)-\nabla f(v)\| \leq L_{\nabla f} \|u-v\| \qquad \forall u,v\in\mcX.
\]
\end{assumption}

As remarked in Section \ref{sec:storm},  the average smoothness condition implies the Lipschitz smoothness condition, but the reverse implication generally does not hold. Therefore, Assumption \ref{a-mm} is weaker than Assumption \ref{a-storm} in general. 

We are now ready to present the convergence results for Algorithm \ref{alg2},  with the proof deferred to Subsection \ref{sec:proof-mm}. Specifically, we will establish convergence rates for the quantities introduced in \eqref{stationarity-criteria}.

\begin{thm}\label{thm:mm}
Suppose that Assumptions \ref{a1} and \ref{a-mm} hold, and $\{x_k\}$ is generated by Algorithm \ref{alg2}. Let $L$ be defined in \eqref{def-L}, $L_f$, $L_{\nabla f}$, $L_c$, $C_c$, $\sigma$, $\gamma$, $\theta$ and $Q_1^*$ be given in Assumptions \ref{a1} and  \ref{a-mm},  $g_1$  be given in Algorithm \ref{alg2}, and $\iota_k$ be the random variable uniformly generated in $\left\{\lceil k/2\rceil+1,\dots,k\right\}$ for $k\geq 2$. Then the following statements hold.
\begin{enumerate}[label=(\roman*)]
\item Suppose that $\theta\in[1,2)$  and its actual value is known. Let $\rho_k$, $\eta_k$ and $\alpha_k$ be  chosen as
\beq\label{def-para3}
\rho_k=k^{\frac{\theta}{4}},\quad \eta_k=k^{-\frac{1}{2}}/\log(k+2),\quad \alpha_k=k^{-\frac{1}{2}}.
\eeq
Then for all $k\geq2\tK_2$, we have
\begin{align*}
&\bbE\left[\dist^2\left(0,\nabla f(x_{\iota_k})+\rho_{\iota_k-1}\nabla c(x_{\iota_k})c(x_{\iota_k})+\mcN_\mcX(x_{\iota_k})\right)\right]\\
&\leq\frac{51\log(k+2)}{2(k-1)^{\frac{1}{2}}}\Bigg( f(x_1)+\frac{1}{2}\|c(x_1)\|^2-Q_1^*+\|g_1-\nabla f(x_1)\|^2+\frac{\theta(6-\theta) }{4(2-\theta)}C_2+\frac{1}{2}(\tK_2^{\frac{\theta}{4}}-1)C_c^2\nn\\
&\quad+\sigma^2(1+\log k)+(1+\log\tK_2)\big(L_{\nabla f}+\tK_2^{\frac{\theta}{4}}L+2\tK_2^{\frac{1}{2}}L_{\nabla f}^2\big)\big(L_f^2+C_c^2L_c^2\tK_2^{\frac{\theta}{2}}\big)\Bigg),\\
&\|c(x_{\iota_k})\|^2\leq2C_2(k/2)^{-\frac{1}{2}},
\end{align*}
where
\begin{align}
&\tK_2=\left\lceil\max\left\{e^2,64L_{\nabla f}^2,e^{8L_{\nabla f}^2},(8L)^{\frac{4}{2-\theta}},\left(e^{-1}\gamma^{-2}2^{2-\frac{\theta}{2}}\log(e^2+2)\right)^{\frac{4}{2-\theta}}\right\}\right\rceil,\label{def-tK3}\\
&C_2=\max\left\{1,\tK_2^{1/2}C_c^2/2,2^{2-\theta/2}L_f^2\gamma^{-2}\right\}.\label{def-C3}
\end{align}

\item Suppose that $\theta\geq1$. Let $\rho_k$, $\eta_k$ and $\alpha_k$ be chosen as
\beq\label{def-para4}
\rho_k=k^{\frac{1}{2}},\quad\eta_k=k^{-\frac{1}{2}}/(4\log(k+2)), \quad \alpha_k=k^{-\frac{1}{2}}.
\eeq
Then for all $k\geq2\tK_3$, we have
\begin{align*}
&\bbE\left[\dist^2\left(0,\nabla f(x_{\iota_k})+\rho_{\iota_k-1}\nabla c(x_{\iota_k})c(x_{\iota_k})+\mcN_\mcX(x_{\iota_k})\right)\right]\\
&\leq\frac{102\log(k+2)}{(k-1)^{\frac{1}{2}}}\Bigg( f(x_1)+\frac{1}{2}\|c(x_1)\|^2-Q_1^*+\|g_1-\nabla f(x_1)\|^2+\frac{1}{2} C_3\max\{1,k^{\frac{1}{2}-\frac{1}{\theta}}\}(1+\log k)\nn \\
&\quad+\frac{1}{2}\big(\tK_3^{\frac{1}{2}}-1\big)C_c^2+\sigma^2(1+\log k)+\frac{1}{16}(1+\log\tK_3)\big(L_{\nabla f}+\tK_3^{\frac{1}{2}}L+2\tK_3^{\frac{1}{2}}L_{\nabla f}^2\big)\big(L_f^2+C_c^2L_c^2\tK_3\big)\Bigg),\\
&\|c(x_{\iota_k})\|^2\leq2C_3(k/2)^{-1/\theta},
\end{align*}
where
\begin{align}
&\tK_3=\left\lceil\max\left\{4L_{\nabla f}^2,e^{2L_{\nabla f}^2},e^{2L},e^{2\theta},\left(e^{-1}\gamma^{-2}2^{6-\theta}\log(e^{2\theta}+2)\right)^{2\theta}\right\}\right\rceil,\label{def-tK4}\\
&C_3=\max\left\{1,\tK_3^{1/\theta}C_c^2/2,2^{3-\theta}L_f^2\gamma^{-2}\right\}.\label{def-nuC4}
\end{align}
\end{enumerate}
\end{thm}

\begin{rem}
For $\theta \in [1,2)$, the choices of $\rho_k$, $\eta_k$, and $\alpha_k$ provided in \eqref{def-para3} and \eqref{def-para4} guarantee the same order of convergence rates for the quantities in \eqref{stationarity-criteria}, regardless of whether the actual value of $\theta$ is known. However, the constant $\tK_2$ generally depends less on $L$ compared to $\tK_3$. Therefore, when $\theta \in [1,2)$ and its actual value is known, the parameters $\rho_k$, $\eta_k$, and $\alpha_k$ specified in \eqref{def-para3} are typically the better choice.
\end{rem}

The following result is an immediate consequence of Theorem \ref{thm:mm}. It provides iteration complexity results for Algorithm \ref{alg2} to find an $\epsilon$-SFSSP $x_{\iota_k}$ of problem \eqref{prob} that satisfies \eqref{kkt}.

\begin{coro}\label{coro:mm}
Suppose that Assumptions \ref{a1} and \ref{a-mm} hold, and $\{x_k\}$ is generated by Algorithm \ref{alg2}. Let $\theta$ be given in Assumption \ref{a1}, and $\iota_k$  be the random variable uniformly generated from $\left\{\lceil k/2\rceil+1,\dots,k\right\}$ for $k\geq 2$. Then the following statements hold.
\begin{enumerate}[label=(\roman*)]
\item Suppose that $\theta\in[1,2)$  and its actual value is known. Let $\rho_k$, $\eta_k$ and $\alpha_k$ be  chosen as in  \eqref{def-para3}. Then for any $\epsilon>0$, there exists some $T=\tmcO(\epsilon^{-4})$ such that \eqref{kkt} holds for all $k\geq T$.

\item Suppose that $\theta\geq1$.
Let $\rho_k$, $\eta_k$ and $\alpha_k$ be chosen as in  \eqref{def-para4}. Then for any $\epsilon>0$, there exists some $T=\tmcO\left(\epsilon^{-\max\{4,2\theta\}}\right)$ such that \eqref{kkt} holds for all $k\geq T$.
\end{enumerate}
\end{coro}

\begin{rem} 
\begin{enumerate}[label=(\roman*)]
\item Since Algorithm \ref{alg2} requires one sample, one gradient evaluation of $c$, and one gradient evaluation of $\tf$ per iteration, its sample complexity and first-order operation complexity are of the same order as its iteration complexity, which is stated in Corollary \ref{coro:mm}. 
\item For $\theta \in [1,2)$, Algorithm \ref{alg2} with $\rho_k$, $\eta_k$ and $\alpha_k$ chosen as in  \eqref{def-para3} or \eqref{def-para4} achieves a sample and first-order operation complexity of  
$\tmcO(\epsilon^{-4})$ for finding an $\epsilon$-SFSSP of  \eqref{prob}. This complexity matches, up to a logarithmic factor, the complexity achieved by the methods in \cite{gao2024non, ghadimi2013stochastic} for problem \eqref{prob} with $c=0$. However, it is worse than the complexity of 
$\tmcO(\epsilon^{-2-\theta})$ achieved by Algorithm \ref{alg1} with $\rho_k$, $\eta_k$ and $\alpha_k$ chosen as in \eqref{def-para2} with $\hat\theta=\theta$. It should be noted that the latter complexity is obtained under stronger assumptions, since Assumption \ref{a-storm} is more restrictive than Assumption \ref{a-mm}.
\item For $\theta\ge 2$, Algorithm \ref{alg2} with $\rho_k$, $\eta_k$ and $\alpha_k$ chosen as in \eqref{def-para4} achieves a sample and first-order operation complexity of  
$\tmcO\left(\epsilon^{-2\theta}\right)$ for finding an $\epsilon$-SFSSP of  \eqref{prob}. This complexity matches that achieved by Algorithm \ref{alg1} with $\rho_k$, $\eta_k$ and $\alpha_k$ chosen as in \eqref{def-para2} with $\hat\theta=\theta$ or $2$, but under weaker assumptions, since Assumption \ref{a-mm} is less restrictive than Assumption \ref{a-storm}. Moreover, Algorithm \ref{alg2} requires only one gradient evaluation of $\tf$ per iteration, while Algorithm \ref{alg1} requires two.
\end{enumerate}
\end{rem}

\section{Proof of the main results}\label{sec:proof}
In this section we provide a proof of our main results presented in Sections \ref{sec:storm} and \ref{sec:mm}, which are particularly Theorems \ref{thm:storm} and \ref{thm:mm}.

\subsection{Proof of the main result in Section \ref{sec:storm}}\label{sec:proof-storm}

In this subsection, we first establish a convergence rate for the \emph{deterministic} feasibility violation by interpreting Algorithm \ref{alg1} as an inexact projected gradient method applied to the associated feasibility problem (see Lemmas \ref{l-rec} and \ref{l-cnst2}). This result, combined with several technical lemmas and a carefully constructed potential function, is then used to prove Theorem \ref{thm:storm}.

For notational convenience, we define
\beq\label{def-h}
h(x) := \frac{1}{2}\|c(x)\|^2.
\eeq
One can observe from Assumption \ref{a1}(iv) that $h$ is $L$-smooth on $\mcX$, where $L$ is given in  \eqref{def-L}.

The following lemma establishes a relationship between $h(x_{k+1})$ and $h(x_k)$, which will  be used to derive bounds for $\|c(x_k)\|^2$.

\begin{lemma}\label{l-rec} 
Suppose that Assumption \ref{a1} holds, and $x_{k+1}$ is generated by Algorithm \ref{alg1} for some $k\geq1$ with $\rho_k\eta_k\leq(\sqrt{5}-1)/(2L)$.  Then we have
\[
h(x_{k+1})+2^{\theta-2}\gamma^2\rho_k\eta_k[h(x_{k+1})]^{\theta}\leq h(x_k)+L_f^2\rho_k^{-1}\eta_k/2,
\]
where $\rho_k$ and $\eta_k$ are given in Algorithm \ref{alg1}, $L_f$, $\gamma$ and $\theta$ are given in Assumption \ref{a1},  and $L$ and $h$ are defined in \eqref{def-L} and \eqref{def-h}, respectively.
\end{lemma}

\begin{proof}
Let $G_k$ be given in Algorithm \ref{alg1}. For convenience, we define
\beq\label{l3-e1}
\widetilde G_k=\rho_k^{-1}G_k,\quad\tilde\eta_k=\rho_k\eta_k.
\eeq
 It then follows from these, \eqref{def-h}, and the expression of $x_{k+1}$ in Algorithm \ref{alg1} that
\beq\label{l3-x}
x_{k+1}=\Pi_\mcX(x_k-\eta_kG_k)=\Pi_\mcX(x_k-\tilde\eta_k\widetilde G_k),
\eeq
which implies that
\beq \label{subprob-opt1}
0\in x_{k+1}-x_k+\tilde\eta_k\widetilde G_k+\mcN_\mcX(x_{k+1})\quad\Rightarrow\quad \nabla h(x_{k+1})+\tilde\eta_k^{-1}(x_k-x_{k+1})-\widetilde G_k\in\nabla h(x_{k+1})+\mcN_\mcX(x_{k+1}).
\eeq
Using this, \eqref{def-h} and Assumption \ref{a1}(iv), we have
\begin{align}
2^\theta\gamma^2[h(x_{k+1})]^\theta&\ =\gamma^2\|c(x_{k+1})\|^{2\theta}\leq\dist^2\left(0,\nabla c(x_{k+1})c(x_{k+1})+\mcN_\mcX(x_{k+1})\right) \nn \\
&\overset{\eqref{def-h}}{=}\dist^2\left(0,\nabla h(x_{k+1})+\mcN_\mcX(x_{k+1})\right)\overset{\eqref{subprob-opt1}}\leq\|\nabla h(x_{k+1})+\tilde\eta_k^{-1}(x_k-x_{k+1})-\widetilde G_k\|^2\nn\\
&\ \leq2\|\tilde\eta_k^{-1}(x_k-x_{k+1})+\nabla h(x_k)-\widetilde G_k\|^2+2\|\nabla h(x_{k+1})-\nabla h(x_k)\|^2\nn\\
&\ =2\tilde\eta_k^{-2}\|x_{k+1}-x_k\|^2+4\tilde\eta_k^{-1}\langle\widetilde G_k-\nabla h(x_k), x_{k+1}-x_k\rangle+2\|\widetilde G_k-\nabla h(x_k)\|^2\nn\\
&\ \quad +2\|\nabla h(x_{k+1})-\nabla h(x_k)\|^2\nn\\
&\ \leq2(\tilde\eta_k^{-2}+L^2)\|x_{k+1}-x_k\|^2+4\tilde\eta_k^{-1}\langle\widetilde G_k-\nabla h(x_k), x_{k+1}-x_k\rangle+2\|\widetilde G_k-\nabla h(x_k)\|^2,\label{l3-e2}
\end{align}
where the first inequality follows from Assumption \ref{a1}(iv), the second inequality is due to the convexity of $\|\cdot\|^2$, and the last inequality follows from the $L$-smoothness of $h$. In addition,  by \eqref{l3-x} and $x_k\in\mcX$, one has
\[
\langle x_{k+1}-x_k+\tilde\eta_k\widetilde G_k,x_k-x_{k+1}\rangle\geq0\quad\Rightarrow\quad \langle \widetilde G_k,x_{k+1}-x_k\rangle\leq-\tilde\eta_k^{-1}\|x_{k+1}-x_k\|^2.
\]
This together with the $L$-smoothness of $h$ yields
\begin{align*}
h(x_{k+1})& \leq h(x_k)+\langle\nabla h(x_k),x_{k+1}-x_k\rangle +\frac{L}{2}\|x_{k+1}-x_k\|^2\nn\\
&=h(x_k)+\langle\widetilde G_k,x_{k+1}-x_k\rangle+\langle\nabla h(x_k)-\widetilde G_k,x_{k+1}-x_k\rangle +\frac{L}{2}\|x_{k+1}-x_k\|^2\nn\\
& \leq h(x_k)-\tilde\eta_k^{-1}\|x_{k+1}-x_k\|^2+\langle\nabla h(x_k)-\widetilde G_k,x_{k+1}-x_k\rangle +\frac{L}{2}\|x_{k+1}-x_k\|^2.
\end{align*}
Using this and \eqref{l3-e2}, we obtain that
\beq\label{l3-e3}
h(x_{k+1})+2^{\theta-2}\gamma^2\tilde\eta_k[h(x_{k+1})]^{\theta}\leq h(x_k)+\frac{1}{2}(L^2\tilde\eta_k-\tilde\eta_k^{-1}+L)\|x_{k+1}-x_k\|^2+\frac{\tilde\eta_k}{2}\|\widetilde G_k-\nabla h(x_k)\|^2.
\eeq
Observe from \eqref{l3-e1} and $\rho_k\eta_k\leq(\sqrt{5}-1)/(2L)$ that $L\tilde\eta_k=L\rho_k\eta_k\leq(\sqrt{5}-1)/2$, which implies that
\beq\label{l3-e4}
L^2\tilde\eta_k-\tilde\eta_k^{-1}+L=\tilde\eta_k^{-1}( L^2\tilde\eta_k^2+L\tilde\eta_k-1)\leq0.
\eeq
Notice from the expression of $g_k$ in Algorithm \ref{alg1} that $g_k\in\cB(L_f)$ and hence $\|g_k\|\leq L_f$. 
Also, observe from Algorithm \ref{alg1} and \eqref{def-h} that $G_k=g_k+\rho_k\nabla h(x_k)$. Using these and \eqref{l3-e1}, we have
\[
\|\widetilde G_k-\nabla h(x_k)\|= \|\rho_k^{-1}G_k-\nabla h(x_k)\| =\|\rho_k^{-1}(g_k+\rho_k\nabla h(x_k))-\nabla h(x_k)\|= \rho_k^{-1}\|g_k\|\leq\rho_k^{-1}L_f.
\]
It then follows from this, \eqref{l3-e3} and \eqref{l3-e4} that 
\[
h(x_{k+1})+2^{\theta-2}\gamma^2\tilde\eta_k[h(x_{k+1})]^\theta\leq h(x_k)+L_f^2\tilde\eta_k/(2\rho_k^2).
\]
This and the definition of $\tilde\eta_k$ in \eqref{l3-e1} imply that the conclusion of this lemma holds.
\end{proof}

The next lemma derives a bound for $\|c(x_k)\|^2$ under the choice of  $\rho_k$, $\eta_k$ and $\alpha_k$ in Algorithm \ref{alg1}.

\begin{lemma}\label{l-cnst2}
Let $\nu$, $\tK_1$, and $C_1$ be given in \eqref{def-para2}, \eqref{def-tK2} and \eqref{def-nuC2}, respectively. Suppose that Assumption \ref{a1} holds, and $\{x_k\}$ is generated by Algorithm \ref{alg1} with $\{\rho_k\}$, $\{\eta_k\}$ and $\{\alpha_k\}$ given in \eqref{def-para2}. Then  we have $\|c(x_k)\|^2\leq 2C_1k^{-2\nu/\theta}$ for all $k\geq\tK_1$.
\end{lemma}

\begin{proof}
Let $h$ be defined in \eqref{def-h}. To prove this lemma, it is equivalent to show that $h(x_k)\leq C_1k^{-2\nu/\theta}$ for all $k\geq\tK_1$.  We now prove this by induction. Indeed, notice from Algorithm \ref{alg1} that $x_{\tK_1}\in X$. It then follows from  \eqref{def-nuC2}, \eqref{def-h}, and Assumption \ref{a1}(iv) that
\[
h(x_{\tK_1})\overset{\eqref{def-h}}{=}\frac{1}{2}\|c(x_{\tK_1})\|^2\leq\frac{1}{2}C_c^2\overset{\eqref{def-nuC2}}{\leq} C_1\tK_1^{-2\nu/\theta}.
\]
Hence,  the conclusion holds for $k=\tK_1$. Now, suppose for induction that $h(x_k)\leq C_1k^{-2\nu/\theta}$ holds for some $k\geq\tK_1$. Recall that $\rho_k$, $\eta_k$ and $\tK_1$ are given in \eqref{def-para2} and \eqref{def-tK2}.  In view of these, one can observe that 
\[
\rho_k\eta_k\overset{\eqref{def-para2} }=\frac{1}{4\log(k+2)}\leq\frac{1}{4\log(\tK_1+2)}\overset{\eqref{def-tK2}}\leq \frac{1}{8L}\leq\frac{\sqrt{5}-1}{2L},
\]
and hence Lemma \ref{l-rec} holds for such $k$.  Using Lemma \ref{l-rec} with the choice of $\rho_k$ and $\eta_k$ given in \eqref{def-para2}, we obtain that
\beq\label{rec2}
h(x_{k+1})+2^{\theta-4}\gamma^2[h(x_{k+1})]^{\theta}/\log(k+2)\leq h(x_k)+L_f^2k^{-2\nu}/(8\log(k+2)).
\eeq
Further, let
\beq\label{def-phi2}
\phi(t)=t+2^{\theta-4}\gamma^2t^\theta/\log(k+2).
\eeq
Notice from \eqref{def-para2} and \eqref{def-nuC2} that $\nu=\min\{\hat\theta/(\hat\theta+2),1/2\}$ for some $\hat\theta\geq1$ and $C_1\geq1$. Using these and \eqref{def-phi2},  we have
\begin{align}
&\ \phi(C_1(k+1)^{-2\nu/\theta})-C_1k^{-2\nu/\theta}-L_f^2k^{-2\nu}/(8\log(k+2))\nn\\
&\overset{\eqref{def-phi2}}{=}C_1^\theta 2^{\theta-4}\gamma^2(k+1)^{-2\nu}/\log(k+2)+C_1(k+1)^{-2\nu/\theta}-C_1k^{-2\nu/\theta}-L_f^2k^{-2\nu}/(8\log(k+2))\nn\\
&\ \geq C_1^\theta 2^{\theta-4}\gamma^2(k+1)^{-2\nu}/\log(k+2)-2\nu C_1 k^{-2\nu/\theta-1}/\theta-L_f^2k^{-2\nu}/(8\log(k+2))\nn\\
&\ = \frac{k^{-2\nu}}{\log(k+2)}\left(C_1^\theta 2^{\theta-4}\gamma^2\left(\frac{k}{k+1}\right)^{2\nu}- 2\nu C_1 k^{2\nu(1-\theta^{-1})-1}\log(k+2)/\theta-L_f^2/8\right)\nn\\
&\ \geq \frac{k^{-2\nu}}{\log(k+2)}\left(C_1^\theta 2^{\theta-5}\gamma^2-C_1 k^{-\frac{1}{\theta}}\log(k+2)-L_f^2/8\right)\nn\\
&\ \geq \frac{k^{-2\nu}}{\log(k+2)}\left(C_12^{\theta-5}\gamma^2-C_1 k^{-\frac{1}{\theta}}\log(k+2)-L_f^2/8\right),\label{l4-e1g}
\end{align}
where the first inequality follows from $(k+1)^{-2\nu/\theta}-k^{-2\nu/\theta}\geq-2\nu k^{-2\nu/\theta-1}/\theta$ thanks to the convexity of $t^{-2\nu/\theta}$, the second inequality is due to $\theta\geq1$, $\nu\leq1/2$ and  $k/(k+1)\geq1/2$, and the last inequality follows from $\theta\geq1$ and $C_1\geq1$. In addition, one can verify that $t^{-\frac{1}{2\theta}}\log(t+2)$ is decreasing on $[e^{2\theta},\infty)$. Using this, \eqref{def-tK2}, and $k\geq\tK_1\geq e^{2\theta}$, we obtain that
\[
k^{-\frac{1}{2\theta}}\log(k+2)\leq\log(e^{2\theta}+2)/e,\quad k^{-\frac{1}{2\theta}}\leq\tK_1^{-\frac{1}{2\theta}}\leq e\gamma^2/(2^{6-\theta}\log(e^{2\theta}+2)).
\]
Multiplying both sides of these inequalities yields $k^{-1/\theta}\log(k+2)\leq2^{\theta-6}\gamma^2$, which together with \eqref{def-nuC2}  implies that
\begin{align*}
C_12^{\theta-5}\gamma^2-C_1 k^{-\frac{1}{\theta}}\log(k+2)-L_f^2/8\geq C_12^{\theta-6}\gamma^2-L_f^2/8\overset{\eqref{def-nuC2}}{\geq}0.
\end{align*}
Using this,  \eqref{rec2}, \eqref{def-phi2},  \eqref{l4-e1g}, and the induction hypothesis that $h(x_k)\leq C_1k^{-2\nu/\theta}$, we obtain that
\[
\phi(C_1(k+1)^{-2\nu/\theta})\geq C_1k^{-2\nu/\theta}+L_f^2k^{-2\nu}/(8\log(k+2))\geq h(x_k)+L_f^2k^{-2\nu}/(8\log(k+2))\overset{\eqref{rec2}\eqref{def-phi2}}{\geq}\phi(h(x_{k+1})).
\]
It then follows from this inequality and the strict monotonicity of $\phi$ on $[0,\infty)$ that $h(x_{k+1})\leq C_1(k+1)^{-2\nu/\theta}$. Hence, the induction is completed and the conclusion of this lemma holds. 
\end{proof}

The following lemma provides a relationship between $\bbE\left[\|g_{k+1}-\nabla f(x_{k+1})\|^2\right]$ and 
$\bbE\left[\|g_k-\nabla f(x_k)\|^2\right]$.

\begin{lemma}\label{l-exp}
Suppose that Assumptions \ref{a1} and \ref{a-storm} hold, and $\{g_k\}$ and $\{x_k\}$ are generated by Algorithm \ref{alg1}. Then for all $k\geq1$, we have
\[
\bbE\left[\|g_{k+1}-\nabla f(x_{k+1})\|^2\right]\leq(1-\alpha_k)^2\bbE\left[\|g_k-\nabla f(x_k)\|^2\right]+6\bL_{\nabla f}^2\bbE\left[\|x_{k+1}-x_k\|^2\right]+3\sigma^2\alpha_k^2,
\]
where $\{\alpha_k\}$ is given in Algorithm \ref{alg1}, and $\sigma$ and $\bL_{\nabla f}$ are given in Assumptions \ref{a1} and \ref{a-storm}, respectively. 
\end{lemma}

\begin{proof}
Notice from \eqref{cf-bnd} that $\nabla f(x_{k+1})\in\cB(L_f)$ and hence $\nabla f(x_{k+1})=\Pi_{\cB(L_f)}(\nabla f(x_{k+1}))$. By this, the expression of $g_{k+1}$, and the nonexpansiveness of the projection operator $\Pi_{\cB(L_f)}$, one has
\begin{align}
&\|g_{k+1}-\nabla f(x_{k+1})\|^2=\|\Pi_{\cB(L_f)}(\nabla \tf(x_{k+1},\xi_{k+1})+(1-\alpha_k)(g_k-\nabla \tf(x_k,\xi_{k+1})))-\Pi_{\cB(L_f)}(\nabla f(x_{k+1}))\|^2 \nn \\
&\leq\|\nabla \tf(x_{k+1},\xi_{k+1})+(1-\alpha_k)(g_k-\nabla \tf(x_k,\xi_{k+1}))-\nabla f(x_{k+1})\|^2 \nn \\
&=\|\nabla \tf(x_{k+1},\xi_{k+1})-\nabla f(x_{k+1})+(1-\alpha_k)(g_k-\nabla f(x_k)+\nabla f(x_k)-\nabla \tf(x_k,\xi_{k+1}))\|^2 \nn \\
&=\|\nabla \tf(x_{k+1},\xi_{k+1})-\nabla f(x_{k+1})+(1-\alpha_k)(\nabla f(x_k)-\nabla \tf(x_k,\xi_{k+1}))\|^2+(1-\alpha_k)^2\|g_k-\nabla f(x_k)\|^2\nn \\
&\quad+2(1-\alpha_k)\langle g_k-\nabla f(x_k),\nabla \tf(x_{k+1},\xi_{k+1})-\nabla f(x_{k+1})\rangle\nn \\
&\quad+2(1-\alpha_k)^2\langle g_k-\nabla f(x_k),\nabla f(x_k)-\nabla \tf(x_k,\xi_{k+1})\rangle. \label{g-bnd}
\end{align}
Let $\Xi_k=\{\xi_1,\ldots, \xi_k\}$ denote the collection of samples drawn up to iteration $k-1$ in Algorithm 
\ref{alg1}.  It then follows from Assumption \ref{a1}(iii) that 
\[
\bbE[\nabla \tf(x_{k+1},\xi_{k+1})-\nabla f(x_{k+1})|\Xi_k]=0, \quad \bbE[\nabla f(x_k)-\nabla \tf(x_k,\xi_{k+1})|\Xi_k]=0,
\]
which imply that
\begin{align*}
&\bbE[\langle g_k-\nabla f(x_k),\nabla \tf(x_{k+1},\xi_{k+1})-\nabla f(x_{k+1})\rangle|\Xi_k]=\langle g_k-\nabla f(x_k),\bbE[\nabla \tf(x_{k+1},\xi_{k+1})-\nabla f(x_{k+1})|\Xi_k]\rangle=0, \\
&\bbE[\langle g_k-\nabla f(x_k),\nabla f(x_k)-\nabla \tf(x_k,\xi_{k+1})\rangle|\Xi_k]=\langle g_k-\nabla f(x_k),\bbE[\nabla f(x_k)-\nabla \tf(x_k,\xi_{k+1})|\Xi_k]\rangle=0.
\end{align*}
Using these and taking a conditional expectation on both sides of \eqref{g-bnd}, we have
\begin{align}
\bbE[\|g_{k+1}-\nabla f(x_{k+1})\|^2|\Xi_k]\leq&\bbE[\|\nabla \tf(x_{k+1},\xi_{k+1})-\nabla f(x_{k+1})+(1-\alpha_k)(\nabla f(x_k)-\nabla \tf(x_k,\xi_{k+1}))\|^2|\Xi_k]\nn\\
&+(1-\alpha_k)^2\|g_k-\nabla f(x_k)\|^2.\label{l1-e1}
\end{align}
In addition, it follows from Assumption \ref{a1} that
\begin{align*}
&\bbE[\|\nabla \tf(x_{k+1},\xi_{k+1})-\nabla f(x_{k+1})+(1-\alpha_k)(\nabla f(x_k)-\nabla \tf(x_k,\xi_{k+1}))\|^2|\Xi_k]\\
&=\bbE[\|\nabla \tf(x_{k+1},\xi_{k+1})-\nabla \tf(x_k,\xi_{k+1})+\nabla f(x_k)-\nabla f(x_{k+1})-\alpha_k(\nabla f(x_k)-\nabla \tf(x_k,\xi_{k+1}))\|^2|\Xi_k]\\
&\leq3\bbE[\|\nabla \tf(x_{k+1},\xi_{k+1})-\nabla \tf(x_k,\xi_{k+1})\|^2|\Xi_k]+3\|\nabla f(x_{k+1})-\nabla f(x_k)\|^2\\
&\quad+3\alpha_k^2\bbE[\|\nabla f(x_k)-\nabla \tf(x_k,\xi_{k+1})\|^2|\Xi_k]\leq6\bL_{\nabla f}^2\|x_{k+1}-x_k\|^2+3\sigma^2\alpha_k^2,
\end{align*}
where the first inequality follows from the convexity of $\|\cdot\|^2$, and the last inequality is due to \eqref{bL-smooth} and Assumptions \ref{a1}(iii) and \ref{a-storm}. By this and \eqref{l1-e1}, one has
\[
\bbE[\|g_{k+1}-\nabla f(x_{k+1})\|^2|\Xi_k]\leq (1-\alpha_k)^2\|g_k-\nabla f(x_k)\|^2+6\bL_{\nabla f}^2\|x_{k+1}-x_k\|^2+3\sigma^2\alpha_k^2.
\]
The conclusion of this lemma follows from taking expectation on both sides of this inequality.
\end{proof}

The next lemma provides an upper bound on $\bbE[Q_{\rho_k}(x_k)+\zeta_k\|g_k-\nabla f(x_k)\|^2]$,  where $\zeta_k$ is given in \eqref{def-zeta}.

\begin{lemma}\label{l-pfunc}
Suppose that Assumptions \ref{a1} and \ref{a-storm} hold. Let $\{g_k\}$ and $\{x_k\}$ be generated by Algorithm \ref{alg1}, and  let 
\beq\label{def-zeta}
\zeta_k=k^{\nu}/(2\log 3),
\eeq
where $\nu$ is defined in \eqref{def-para2}. Then for all $k\geq1$, we have
\begin{align}
&\bbE\left[Q_{\rho_k}(x_k)+\zeta_k\|g_k-\nabla f(x_k)\|^2\right] \leq Q_{\rho_1}(x_1)+\zeta_1\|g_1-\nabla f(x_1)\|^2+\frac{1}{2}\sum_{i=1}^{k-1}(\rho_{i+1}-\rho_i)\bbE\left[\|c(x_{i+1})\|^2\right] \nn \\
&+\frac{1}{2}\sum_{i=1}^{k-1}\left(\bL_{\nabla f}+\rho_iL-\eta_i^{-1}+12\zeta_{i+1}\bL_{\nabla f}^2\right)\bbE\left[\|x_{i+1}-x_i\|^2\right]+3\sigma^2\sum_{i=1}^{k-1}\zeta_{i+1}\alpha_i^2, \label{penalty-bnd}
\end{align}
where $\{\alpha_k\}$, $\{\rho_k\}$ and $\{\eta_k\}$ are given in Algorithm \ref{alg1}, $Q_\rho$  and $L$ are respectively defined in \eqref{def-Q} and \eqref{def-L},  and $\sigma$ and $\bL_{\nabla f}$  are given in Assumptions \ref{a1} and \ref{a-storm}, respectively.
\end{lemma}

\begin{proof}
Observe from Assumptions \ref{a1} and \ref{a-storm} and the definition of $Q_\rho$ in \eqref{def-Q} that $Q_{\rho_k}$ is $(\bL_{\nabla f}+\rho_kL)$-smooth on $\mcX$, where $L$ is defined in \eqref{def-L}. Notice from Algorithm \ref{alg1} that $x_k\in\mcX$ and $x_{k+1}= \Pi_X(x_k-\eta_k G_k)$, which imply that
\beq \label{subprob-opt2}
\langle x_{k+1}-x_k+\eta_k G_k,x_k-x_{k+1}\rangle\geq0\quad\Rightarrow\quad \langle G_k,x_{k+1}-x_k\rangle\leq-\eta_k^{-1}\|x_{k+1}-x_k\|^2.
\eeq
Also, notice from Algorithm \ref{alg1} and \eqref{def-Q} that 
\[
G_k=g_k+\rho_k\nabla c(x_k)c(x_k), \quad \nabla Q_{\rho_k}(x_k)=\nabla f(x_k)+\rho_k\nabla c(x_k)c(x_k),
\] 
and hence $\nabla Q_{\rho_k}(x_k)-G_k=g_k-\nabla f(x_k)$.  In addition, by Young's inequality,  one has 
\beq \label{young-ineq}
\langle\nabla Q_{\rho_k}(x_k)-G_k,x_{k+1}-x_k\rangle \leq \frac{1}{2\eta_k}\|x_{k+1}-x_k\|^2+\frac{\eta_k}{2}\|\nabla Q_{\rho_k}(x_k)-G_k\|^2
\eeq
Using the last two relations,  \eqref{subprob-opt2}, and the $(\bL_{\nabla f}+\rho_kL)$-smoothness of $Q_{\rho_k}$, we obtain that
\begin{align*}
Q_{\rho_k}(x_{k+1})&\ \leq  Q_{\rho_k}(x_k)+\langle\nabla Q_{\rho_k}(x_k),x_{k+1}-x_k\rangle+\frac{1}{2}(\bL_{\nabla f}+\rho_kL)\|x_{k+1}-x_k\|^2\\
&\ = Q_{\rho_k}(x_k)+\langle G_k,x_{k+1}-x_k\rangle+\langle\nabla Q_{\rho_k}(x_k)-G_k,x_{k+1}-x_k\rangle+\frac{1}{2}(\bL_{\nabla f}+\rho_kL)\|x_{k+1}-x_k\|^2\\
& \overset{\eqref{young-ineq}}{\leq} Q_{\rho_k}(x_k)+\langle G_k,x_{k+1}-x_k\rangle+\frac{1}{2}\left(\bL_{\nabla f}+\rho_kL+\eta_k^{-1}\right)\|x_{k+1}-x_k\|^2+\frac{\eta_k}{2}\|\nabla Q_{\rho_k}(x_k)-G_k\|^2\\
&\ \leq Q_{\rho_k}(x_k)+\frac{1}{2}\left(\bL_{\nabla f}+\rho_kL-\eta_k^{-1}\right)\|x_{k+1}-x_k\|^2+\frac{\eta_k}{2}\|g_k-\nabla f(x_k)\|^2,
\end{align*}
where the first inequality is due to the $(\bL_{\nabla f}+\rho_kL)$-smoothness of $Q_{\rho_k}$, and the last inequality follows from \eqref{subprob-opt2} and the relation $\nabla Q_{\rho_k}(x_k)-G_k=g_k-\nabla f(x_k)$. By this and \eqref{def-Q}, we further have
\begin{align}
&Q_{\rho_{k+1}}(x_{k+1})\leq Q_{\rho_k}(x_k)+\frac{1}{2}\left(\bL_{\nabla f}+\rho_kL-\eta_k^{-1}\right)\|x_{k+1}-x_k\|^2+\frac{\eta_k}{2}\|g_k-\nabla f(x_k)\|^2\nn\\
&\qquad\qquad\qquad\,\,+Q_{\rho_{k+1}}(x_{k+1})-Q_{\rho_k}(x_{k+1})\nn\\
&\overset{\eqref{def-Q}}{=} Q_{\rho_k}(x_k)+\frac{1}{2}\left(\bL_{\nabla f}+\rho_kL-\eta_k^{-1}\right)\|x_{k+1}-x_k\|^2+\frac{\eta_k}{2}\|g_k-\nabla f(x_k)\|^2+\frac{1}{2}(\rho_{k+1}-\rho_k)\|c(x_{k+1})\|^2.\label{l5-e1}
\end{align}

By the definitions of $\eta_k$, $\alpha_k$, $\nu$ and $\zeta_k$ in \eqref{def-para2}, \eqref{def-nuC2} and \eqref{def-zeta}, one has
\begin{align*}
&\ \zeta_k-\zeta_{k+1}(1-\alpha_k)^2-\eta_k=\zeta_{k+1}\alpha_k(2-\alpha_k)-\zeta_{k+1}+\zeta_k-\eta_k\geq \zeta_{k+1}\alpha_k-\zeta_{k+1}+\zeta_k-\eta_k\\
&\overset{\eqref{def-para2}\eqref{def-zeta}}\geq(k+1)^\nu k^{-2\nu}/(2\log 3)-(k+1)^\nu/(2\log 3)+k^\nu/(2\log 3)-k^{-\nu}/(4\log(k+2))\\
&\ \ \ \geq k^{-\nu}/(2\log 3)-(k+1)^\nu/(2\log 3)+k^\nu/(2\log 3)-k^{-\nu}/(4\log(k+2))\\
&\ \ \ \geq k^{-\nu}/(2\log 3)-\nu k^{\nu-1}/(2\log 3)-k^{-\nu}/(4\log(k+2))\\
&\ \ \ = \frac{k^{-\nu}}{4\log 3}\left(2-2\nu k^{2\nu-1}-\log 3/\log(k+2)\right) \geq\frac{k^{-\nu}}{4\log 3}\left(2-1-\log 3/\log(k+2)\right)\geq0,
\end{align*}
where the first inequality is due to  $0<\alpha_k\leq1$, the third inequality follows from $(k+1)^\nu>k^\nu$, the fourth inequality is due to $(k+1)^\nu-k^\nu\leq \nu k^{\nu-1}$ thanks to the concavity of $t^\nu$, the fifth inequality follows from $\nu\leq1/2$, and the last inequality is due to $k\geq1$. Hence, we obtain
\[
\zeta_{k+1}(1-\alpha_k)^2+\eta_k\leq\zeta_k\qquad \forall k\geq1.
\]
Using this, taking expectation on both sides of \eqref{l5-e1}, and summing the resulting inequality with the inequality in Lemma \ref{l-exp}, we obtain that
\begin{align}
&\bbE\left[Q_{\rho_{k+1}}(x_{k+1})+\zeta_{k+1}\|g_{k+1}-\nabla f(x_{k+1})\|^2\right]\nn\\
&\leq \bbE\left[Q_{\rho_k}(x_k)+\left(\zeta_{k+1}(1-\alpha_k)^2+\eta_k\right)\|g_k-\nabla f(x_k)\|^2\right]+\frac{1}{2}\left(\bL_{\nabla f}+\rho_kL-\eta_k^{-1}+12\zeta_{k+1}\bL_{\nabla f}^2\right) \nn \\ 
&\quad\times \bbE\left[\|x_{k+1}-x_k\|^2\right] -\frac{\eta_k}{2}\bbE\left[\|g_k-\nabla f(x_k)\|^2\right]+\frac{1}{2}(\rho_{k+1}-\rho_k)\bbE\left[\|c(x_{k+1})\|^2\right]+3\sigma^2\zeta_{k+1}\alpha_k^2\nn\\
&\leq \bbE\left[Q_{\rho_k}(x_k)+\zeta_k\|g_k-\nabla f(x_k)\|^2\right]+\frac{1}{2}\left(\bL_{\nabla f}+\rho_kL-\eta_k^{-1}+12\zeta_{k+1}\bL_{\nabla f}^2\right)\bbE\left[\|x_{k+1}-x_k\|^2\right]\nn\\
&\quad-\frac{\eta_k}{2}\bbE\left[\|g_k-\nabla f(x_k)\|^2\right]+\frac{1}{2}(\rho_{k+1}-\rho_k)\bbE\left[\|c(x_{k+1})\|^2\right]+3\sigma^2\zeta_{k+1}\alpha_k^2.\label{l5-e2}
\end{align}
The conclusion of this lemma follows by replacing $k$ with $i$ in the above inequalities and summing them up for all $1\leq i\leq k-1$.
\end{proof}

The following lemma provides an upper bound on $\dist^2\left(0,\nabla Q_{\rho_k}(x_{k+1})+\mcN_\mcX(x_{k+1})\right)$.

\begin{lemma}\label{l-dist}
Suppose that Assumptions \ref{a1} and \ref{a-storm} hold, and $\{g_k\}$ and  $\{x_k\}$ are generated by Algorithm \ref{alg1}. Then for all $k\geq1$, we have
\beq \label{residual-bnd}
\dist^2\left(0,\nabla Q_{\rho_k}(x_{k+1})+\mcN_\mcX(x_{k+1})\right)\leq3\left(\eta_k^{-2}+(\bL_{\nabla f}+\rho_kL)^2\right)\|x_{k+1}-x_k\|^2+3\|g_k-\nabla f(x_k)\|^2,
\eeq
where  $\{\rho_k\}$ and $\{\eta_k\}$ are given in Algorithm \ref{alg1}, $\bL_{\nabla f}$ is given in Assumption \ref{a-storm},  and $L$ and $Q_\rho$ are defined in \eqref{def-L} and \eqref{def-Q}, respectively.
\end{lemma}

\begin{proof}
By the expression of $x_{k+1}$ in Algorithm \ref{alg1}, one has
\beq \label{subprob_opt}
0\in x_{k+1}-x_k+\eta_k G_k+\mcN_\mcX(x_{k+1})\quad\Rightarrow\quad\eta_k^{-1}(x_k-x_{k+1})-G_k\in\mcN_\mcX(x_{k+1}).
\eeq
Notice from the definition of $Q_\rho$ in \eqref{def-Q} that $\nabla Q_{\rho_k}(x)=\nabla f(x)+\rho_k\nabla c(x)c(x)$, which together with \eqref{def-L}, \eqref{bL-smooth} and Assumption \ref{a1}(iv) implies that $Q_{\rho_k}$ is $(\bL_{\nabla f}+\rho_kL)$-smooth on $\mcX$. Using \eqref{subprob_opt}, the expression of $\nabla Q_{\rho_k}$, and $G_k=g_k+\rho_k\nabla c(x_k)c(x_k)$ (see Algorithm \ref{alg1}), we have
\[
\eta_k^{-1}(x_k-x_{k+1})+\nabla f(x_k)-g_k-\nabla Q_{\rho_k}(x_k) = \eta_k^{-1}(x_k-x_{k+1})-g_k-\rho_k\nabla c(x_k)c(x_k) \in \mcN_\mcX(x_{k+1}).
\]
By this and the $(\bL_{\nabla f}+\rho_kL)$-smoothness of $Q_{\rho_k}$, one has
\begin{align*}
&\dist^2\left(0,\nabla Q_{\rho_k}(x_{k+1})+\mcN_\mcX(x_{k+1})\right) \leq \|\nabla Q_{\rho_k}(x_{k+1})
+(\eta_k^{-1}(x_k-x_{k+1})+\nabla f(x_k)-g_k-\nabla Q_{\rho_k}(x_k))\|^2\\
&\leq 3\left(\|\nabla Q_{\rho_k}(x_{k+1})-\nabla Q_{\rho_k}(x_k)\|^2+\eta_k^{-2}\|x_{k+1}-x_k\|^2+\| g_k-\nabla f(x_k)\|^2\right)\\
&\leq 3\left(\eta_k^{-2}+(\bL_{\nabla f}+\rho_kL)^2\right)\|x_{k+1}-x_k\|^2+3\|g_k-\nabla f(x_k)\|^2,
\end{align*}
where the second inequality follows from the convexity of $\|\cdot\|^2$, and the last inequality is due to the $(\bL_{\nabla f}+\rho_kL)$-smoothness of $Q_{\rho_k}$. Hence, the conclusion of this lemma holds.
\end{proof}

We are now ready to prove the main result in Section \ref{sec:storm}, which is particularly Theorem \ref{thm:storm}.

\begin{proof}[\textbf{Proof of Theorem \ref{thm:storm}}]
Using \eqref{def-para2}, \eqref{def-tK2} and \eqref{def-zeta}, we have that for all $i\geq\tK_1$,
\begin{align}
\bL_{\nabla f}+\rho_iL &\overset{\eqref{def-para2}}=\bL_{\nabla f}i^{-\nu}\eta_i^{-1}/(4\log(i+2))+L\eta_i^{-1}/(4\log(i+2))\nn\\
&\ \leq\bL_{\nabla f}\tK_1^{-\nu}\eta_i^{-1}/4+L\eta_i^{-1}/(4\log(\tK_1+2))\overset{\eqref{def-tK2}}\leq\eta_i^{-1}/4,\label{t1-e2}\\
12\zeta_{i+1}\bL_{\nabla f}^2& \overset{\eqref{def-para2}\eqref{def-zeta}}=3\bL_{\nabla f}^2\big((i+1)/i\big)^\nu\eta_i^{-1}/\big(2\log(i+2)\log 3\big)\nn\\
&\ \ \ \leq3\bL_{\nabla f}^22^\nu\eta_i^{-1}/(2\log(\tK_1+2))\overset{\eqref{def-tK2}}\leq\eta_i^{-1}/4,\nn
\end{align}
which imply that 
\beq\label{t1-ineq2}
\bL_{\nabla f}+\rho_iL+12\zeta_{i+1}\bL_{\nabla f}^2\leq\eta_i^{-1}/2\qquad\forall i\geq\tK_1.
\eeq
It then follows from the proof of Lemma \ref{l-pfunc} that \eqref{l5-e2} holds. Using \eqref{t1-ineq2} and rearranging the terms of \eqref{l5-e2} with $k$ replaced by $i$, we obtain that for all $i\geq\tK_1$, 
\begin{align}
&\frac{1}{4\eta_i}\bbE\left[\|x_{i+1}-x_i\|^2\right]+\frac{\eta_i}{2}\bbE\left[\|g_i-\nabla f(x_i)\|^2\right]\nn\\
&\overset{\eqref{l5-e2}}\leq \bbE\left[Q_{\rho_i}(x_i)+\zeta_i\|g_i-\nabla f(x_i)\|^2\right]-\bbE\left[Q_{\rho_{i+1}}(x_{i+1})+\zeta_{i+1}\|g_{i+1}-\nabla f(x_{i+1})\|^2\right]\nn\\
&\quad\,\,\,+\frac{1}{2}\left(\bL_{\nabla f}+\rho_iL+12\zeta_{i+1}\bL_{\nabla f}^2-\eta_i^{-1}/2\right)\bbE\left[\|x_{i+1}-x_i\|^2\right]+\frac{1}{2}(\rho_{i+1}-\rho_i)\bbE\left[\|c(x_{i+1})\|^2\right]+3\sigma^2\zeta_{i+1}\alpha_i^2\nn\\
&\overset{\eqref{t1-ineq2}}{\leq}\bbE\left[Q_{\rho_i}(x_i)+\zeta_i\|g_i-\nabla f(x_i)\|^2\right]-\bbE\left[Q_{\rho_{i+1}}(x_{i+1})+\zeta_{i+1}\|g_{i+1}-\nabla f(x_{i+1})\|^2\right]\nn\\
&\quad\,\,\,+\frac{1}{2}(\rho_{i+1}-\rho_i)\bbE\left[\|c(x_{i+1})\|^2\right]+3\sigma^2\zeta_{i+1}\alpha_i^2.\label{t1-p2}
\end{align}
Recall that $\iota_k$ is the random variable uniformly generated in $\left\{\lceil k/2\rceil+1,\dots,k\right\}$. In addition, observe from \eqref{def-para2} that $\eta_i ^{-1}<\eta_{k-1}^{-1}$ for all $\lceil k/2\rceil \leq i \leq k-1$. By these, \eqref{def-Q}, \eqref{def-para2}, \eqref{penalty-bnd}, \eqref{residual-bnd}, \eqref{t1-e2} and \eqref{t1-p2}, one has that for all $k\geq2\tK_1$,
\begin{align}
&\ \bbE\left[\dist^2\left(0,\nabla f(x_{\iota_k})+\rho_{\iota_k-1}\nabla c(x_{\iota_k})c(x_{\iota_k})+\mcN_\mcX(x_{\iota_k})\right)\right]=\bbE\left[\dist^2\left(0,\nabla Q_{\rho_{\iota_k-1}}(x_{\iota_k})+\mcN_\mcX(x_{\iota_k})\right)\right]\nn\\
&\ =\frac{1}{k-\lceil k/2\rceil}\sum_{i=\lceil k/2\rceil}^{k-1}\bbE\left[\dist^2\left(0,\nabla Q_{\rho_i}(x_{i+1})+\mcN_\mcX(x_{i+1})\right)\right]\nn\\
& \overset{\eqref{residual-bnd}}\leq  \frac{3}{k-\lceil k/2\rceil}\sum_{i=\lceil k/2\rceil}^{k-1}\left(\left(\eta_i^{-2}+(\bL_{\nabla f}+\rho_iL)^2\right)\bbE\left[\|x_{i+1}-x_i\|^2\right]+\bbE\left[\|g_i-\nabla f(x_i)\|^2\right]\right)\nn\\
& \overset{\eqref{t1-e2}}{\leq} \frac{3}{k-\lceil k/2\rceil}\sum_{i=\lceil k/2\rceil}^{k-1}\left((\eta_i^{-2}+\eta_i^{-2}/16)\bbE\left[\|x_{i+1}-x_i\|^2\right]+\bbE\left[\|g_i-\nabla f(x_i)\|^2\right]\right)\nn\\
&\ \leq \frac{51}{8(k-1)}\sum_{i=\lceil k/2\rceil}^{k-1}\left(\eta_i^{-2}\bbE\left[\|x_{i+1}-x_i\|^2\right]+2\bbE\left[\|g_i-\nabla f(x_i)\|^2\right]\right)\nn\\
&\ \leq \frac{51}{2(k-1)\eta_{k-1}}\sum_{i=\lceil k/2\rceil}^{k-1}\left(\frac{1}{4\eta_i}\bbE\left[\|x_{i+1}-x_i\|^2\right]+\frac{\eta_i}{2}\bbE\left[\|g_i-\nabla f(x_i)\|^2\right]\right)\nn\\
& \overset{\eqref{t1-p2}}{\leq}\frac{51}{2(k-1)\eta_{k-1}}\sum_{i=\lceil k/2\rceil}^{k-1}\Big(\bbE\left[Q_{\rho_i}(x_i)+\zeta_i\|g_i-\nabla f(x_i)\|^2\right]-\bbE\left[Q_{\rho_{i+1}}(x_{i+1})+\zeta_{i+1}\|g_{i+1}-\nabla f(x_{i+1})\|^2\right]\nn\\
&\ \quad\,\,\,+\frac{1}{2}(\rho_{i+1}-\rho_i)\bbE\left[\|c(x_{i+1})\|^2\right]+3\sigma^2\zeta_{i+1}\alpha_i^2\Big)\nn\\
&\ = \frac{51}{2(k-1)\eta_{k-1}}\Bigg(\bbE\left[Q_{\rho_{\lceil k/2\rceil}}(x_{\lceil k/2\rceil})+\zeta_{\lceil k/2\rceil}\|g_{\lceil k/2\rceil}-\nabla f(x_{\lceil k/2\rceil})\|^2\right]-\bbE\left[Q_{\rho_k}(x_k)+\zeta_k\|g_k-\nabla f(x_k)\|^2\right]\nn\\
&\ \quad\,\,\,+\frac{1}{2}\sum_{i=\lceil k/2\rceil}^{k-1}(\rho_{i+1}-\rho_i)\bbE\left[\|c(x_{i+1})\|^2\right]+3\sigma^2\sum_{i=\lceil k/2\rceil}^{k-1}\zeta_{i+1}\alpha_i^2\Bigg)\nn\\
&\ \leq \frac{51}{2(k-1)\eta_{k-1}}\Bigg(Q_1(x_1)-Q_1^*+\zeta_1\|g_1-\nabla f(x_1)\|^2+\frac{1}{2}\sum_{i=1}^{k-1}(\rho_{i+1}-\rho_i)\bbE\left[\|c(x_{i+1})\|^2\right]+3\sigma^2\sum_{i=1}^{k-1}\zeta_{i+1}\alpha_i^2\nn\\
&\ \quad\,\,\,+\frac{1}{2}\sum_{i=1}^{\lceil k/2\rceil-1}\left(\bL_{\nabla f}+\rho_iL-\eta_i^{-1}+12\zeta_{i+1}\bL_{\nabla f}^2\right)\bbE\left[\|x_{i+1}-x_i\|^2\right]\Bigg),\label{t1-sum2}
\end{align}
where the first equality is due to \eqref{def-Q}, the first inequality follows from taking expectation on both sides of \eqref{residual-bnd}, the third inequality is due to the fact that $\lceil k/2\rceil\leq(k+1)/2$, the fourth inequality follows from the relation $\eta_i ^{-1}<\eta_{k-1}^{-1}$ for all $\lceil k/2\rceil \leq i \leq k-1$, and the last inequality follows from \eqref{penalty-bnd} with $k$ replaced by $\lceil k/2\rceil$, $\rho_1=1$, and the fact that $Q_{\rho_k}(x_k)\geq Q_1 (x_k) \geq Q_1^*$.

We next bound each term in the summation \eqref{t1-sum2}.  Indeed, it follows from \eqref{cf-bnd}, Assumption \ref{a1}(iv), the nonexpansiveness of $\Pi_X$, and the expressions of $x_{k+1}$, $g_k$ and $G_k$ in Algorithm \ref{alg1} that
\begin{align}
\|x_{k+1}-x_k\|^2&=\|\Pi_\mcX(x_k-\eta_kG_k)-\Pi_X(x_k)\|^2\leq\eta_k^2\|G_k\|^2=\eta_k^2\|g_k+\rho_k\nabla c(x_k)c(x_k)\|^2\nn\\
&\leq2\eta_k^2\left(\|g_k\|^2+\rho_k^2\|\nabla c(x_k)c(x_k)\|^2\right)\leq2\eta_k^2\left(L_f^2+C_c^2L_c^2\rho_k^2\right).\label{t1-x1}
\end{align}
Recall that $\nu=\min\{\hat\theta/(\hat\theta+2),1/2\}$ for some $\hat\theta\geq1$ and $\|c(x_i)\| \leq C_c$ for all $i$. By this, Lemma \ref{l-cnst2}, \eqref{def-para2}, \eqref{def-zeta},  \eqref{t1-x1},  one has that for all $k\geq2\tK_1$, 
\begin{align}
&\sum_{i=1}^{\tK_1-1}(\rho_{i+1}-\rho_i)\bbE\left[\|c(x_{i+1})\|^2\right]\leq C_c^2\sum_{i=1}^{\tK_1-1}\big((i+1)^\nu-i^\nu\big)=C_c^2\big(\tK_1^\nu-1\big),\label{sum2-c1}\\
&\sum_{i=\tK_1}^{k-1}(\rho_{i+1}-\rho_i)\bbE\left[\|c(x_{i+1})\|^2\right]\leq2C_1\sum_{i=\tK_1}^{k-1}((i+1)^\nu-i^\nu)(i+1)^{-\frac{2\nu}{\theta}}\leq 2\nu C_1\sum_{i=\tK_1}^{k-1}i^{\nu-1}(i+1)^{-\frac{2\nu}{\theta}} \label{sum2-c2} \\
&\leq C_1\sum_{i=\tK_1}^{k-1}i^{-1}(i+1)^{\nu-\frac{2\nu}{\theta}}\leq C_1\max\big\{1,k^{\nu-\frac{2\nu}{\theta}}\big\}\sum_{i=\tK_1}^{k-1}i^{-1}\leq C_1\max\big\{1,k^{\nu-\frac{2\nu}{\theta}}\big\}(1+\log k), \label{sum2-c3} \\
&\sum_{i=1}^{k-1}\zeta_{i+1}\alpha_i^2\overset{\eqref{def-zeta}}=\frac{1}{2\log3}\sum_{i=1}^{k-1}(i+1)^{\nu}i^{-4\nu}= \frac{1}{2\log3}\sum_{i=1}^{k-1}i^{-3\nu}((i+1)/i)^{\nu} \nn \\
& \leq\frac{\sqrt{2}}{2\log3}\sum_{i=1}^{k-1}i^{-3\nu} \leq\frac{\sqrt{2}}{2\log3}\sum_{i=1}^{k-1}i^{-1}\leq\frac{\sqrt{2}(1+\log k)}{2\log3},\label{sum2-c4a}\\
&\ \sum_{i=1}^{\lceil k/2\rceil-1}\left(\bL_{\nabla f}+\rho_iL-\eta_i^{-1}+12\zeta_{i+1}\bL_{\nabla f}^2\right)\bbE\left[\|x_{i+1}-x_i\|^2\right]\nn\\
&\ =\sum_{i=1}^{\tK_1-1}\left(\bL_{\nabla f}+\rho_iL-\eta_i^{-1}+12\zeta_{i+1}\bL_{\nabla f}^2\right)\bbE\left[\|x_{i+1}-x_i\|^2\right]+\sum_{i=\tK_1}^{\lceil k/2\rceil-1}\Big((\bL_{\nabla f}+\rho_iL-\eta_i^{-1}+12\zeta_{i+1}\bL_{\nabla f}^2)\nn \\ 
&\ \quad\,\,\,\times \bbE\left[\|x_{i+1}-x_i\|^2\right]\Big) \nn \\
&\overset{\eqref{t1-ineq2}}{\leq}\sum_{i=1}^{\tK_1-1}\left(\bL_{\nabla f}+\rho_iL-\eta_i^{-1}+12\zeta_{i+1}\bL_{\nabla f}^2\right)\bbE\left[\|x_{i+1}-x_i\|^2\right] \nn\\
&\overset{\eqref{t1-x1}}{\leq}2\sum_{i=1}^{\tK_1-1}\left(\bL_{\nabla f}+\rho_iL+12\zeta_{i+1}\bL_{\nabla f}^2\right)\eta_i^2\left(L_f^2+C_c^2L_c^2\rho_i^2\right)\nn\\
&\overset{\eqref{def-para2}}{\leq}\frac{1}{8}(1+\log\tK_1)\big(\bL_{\nabla f}+\tK_1^{\frac{1}{2}}L+6\tK_1^{\frac{1}{2}}\bL_{\nabla f}^2\big)(L_f^2+C_c^2L_c^2\tK_1),\label{sum2-c4}
\end{align}
where the inequality in \eqref{sum2-c1} follows from \eqref{def-para2}, $\|c(x_i)\| \leq C_c$ for all $i$, the inequalities in \eqref{sum2-c2} are due to \eqref{def-para2}, Lemma \ref{l-cnst2},  and $(i+1)^\nu-i^\nu\leq \nu i^{\nu-1}$ for all $i\geq1$ thanks to the concavity of $t^\nu$, the inequalities in \eqref{sum2-c4a} are due to $((i+1)/i)^\nu\leq\sqrt{2}$ for all $i\geq1$ and $\nu\geq1/3$, and the last inequality in \eqref{sum2-c4} is due to the relations $\rho_i \leq \tK_1^{1/2}$ and $\zeta_i \leq \tK_1^{1/2}/2$ for $1\leq i\leq \tK_1-1$ and
$\sum_{i=1}^{\tK_1-1}\eta_i^2 \leq \sum_{i=1}^{\tK_1-1}i^{-1}/16\leq (1+\log\tK_1)/16$ thanks to $\nu\leq 1/2$ and the choice of $\rho_i$, $\eta_i$ and $\zeta_i$ in \eqref{def-para2} and \eqref{def-zeta}.

Using \eqref{def-para2}, \eqref{def-zeta}, \eqref{t1-sum2}, \eqref{sum2-c1}, \eqref{sum2-c3}, \eqref{sum2-c4a} and \eqref{sum2-c4}, we have
\begin{align*}
&\bbE\left[\dist^2\left(0,\nabla f(x_{\iota_k})+\rho_{\iota_k-1}\nabla c(x_{\iota_k})c(x_{\iota_k})+\mcN_\mcX(x_{\iota_k})\right)\right]\\
&\leq \frac{102\log(k+1)}{(k-1)^{1-\nu}}\Bigg(Q_1(x_1)-Q_1^*+\frac{\|g_1-\nabla f(x_1)\|^2}{2\log 3}+\frac{1}{2}C_1\max\{1,k^{\nu-\frac{2\nu}{\theta}}\}(1+\log k)\nn \\ 
&\quad+\frac{1}{2}C_c^2\big(\tK_1^{\nu}-1\big)
+\frac{3\sqrt{2}\sigma^2(1+\log k)}{2\log3}+\frac{1}{16}(1+\log\tK_1)\big(\bL_{\nabla f}+\tK_1^{\frac{1}{2}}L+6\tK_1^{\frac{1}{2}}\bL_{\nabla f}^2\big)(L_f^2+C_c^2L_c^2\tK_1)\Bigg).
\end{align*}
By this, \eqref{def-Q}, $\iota_k>\lceil k/2\rceil\geq\tK_1$ for all $k\geq2\tK_1$, and Lemma \ref{l-cnst2} with $k$ replaced by $\iota_k$, one can see that Theorem \ref{thm:storm} holds.
\end{proof}

\subsection{Proof of the main result in Section \ref{sec:mm}}\label{sec:proof-mm}

In this subsection, we first establish a convergence rate for the  \emph{deterministic} feasibility violation by interpreting Algorithm \ref{alg2} as an inexact projected gradient method applied to the associated feasibility problem (see Lemmas \ref{l-rec2}, \ref{l-cnst3} and \ref{l-cnst4}). This result, together with several technical lemmas and a carefully constructed potential function, is then used to prove Theorem \ref{thm:mm}.

The following lemma establishes a relationship between $h(x_{k+1})$ and $h(x_k)$, which will  be used to derive bounds for $\|c(x_k)\|^2$, where $h$ is defined in \eqref{def-h}.

\begin{lemma}\label{l-rec2} 
Suppose that Assumption \ref{a1} holds, and $x_{k+1}$ is generated by Algorithm \ref{alg2} for some $k\geq1$ with $\rho_k\eta_k\leq(\sqrt{5}-1)/(2L)$.  Then we have
\[
h(x_{k+1})+2^{\theta-2}\gamma^2\rho_k\eta_k[h(x_{k+1})]^{\theta}\leq h(x_k)+L_f^2\rho_k^{-1}\eta_k/2, 
\]
where $\rho_k$ and $\eta_k$ are given in Algorithm \ref{alg2}, $L_f$, $\gamma$ and $\theta$ are given in Assumption \ref{a1},  and $L$ and $h$ are defined in \eqref{def-L} and \eqref{def-h}, respectively.
\end{lemma}

\begin{proof}
The proof of this lemma follows from similar arguments as in the proof of Lemma \ref{l-rec}.
\end{proof}

The next two lemmas derive bounds for $\|c(x_k)\|^2$ under two different choices of  $\rho_k$, $\eta_k$ and $\alpha_k$ in Algorithm \ref{alg2}.

\begin{lemma}\label{l-cnst3}
Let $\tK_2$ and $C_2$ be given in \eqref{def-tK3} and \eqref{def-C3},  respectively. 
Suppose that Assumption \ref{a1} holds with $\theta\in[1,2)$ and  $\{x_k\}$ is generated by Algorithm \ref{alg2} with $\{\rho_k\}$, $\{\eta_k\}$ and $\{\alpha_k\}$ given in \eqref{def-para3}. Then  we have $\|c(x_k)\|^2\leq 2C_2k^{-1/2}$ for all $k\geq\tK_2$.
\end{lemma}

\begin{proof}
Let $h$ be defined in \eqref{def-h}. To prove this lemma, it is equivalent to show that $h(x_k)\leq C_2k^{-1/2}$ for all $k\geq\tK_2$.  We now prove this by induction. Indeed, notice from Algorithm \ref{alg2} that $x_{\tK_2}\in X$. It then follows from \eqref{def-C3}, \eqref{def-h} and Assumption \ref{a1}(iv) that
\[
h(x_{\tK_2})\overset{\eqref{def-h}}{=}\frac{1}{2}\|c(x_{\tK_2})\|^2\leq\frac{1}{2}C_c^2\overset{\eqref{def-C3}}{\leq} C_2\tK_2^{-1/2}.
\]
Hence,  the conclusion holds for $k=\tK_2$. Now, suppose for induction that $h(x_k)\leq C_2k^{-1/2}$ holds for some $k\geq\tK_2$.  Recall that $\theta\in[1,2)$ and $\rho_k$, $\eta_k$ and $\tK_2$ are given in \eqref{def-para3} and \eqref{def-tK3}.  In view of these, one can observe that 
\[
\rho_k\eta_k\overset{\eqref{def-para3}}=\frac{k^{\frac{\theta-2}{4}}}{\log(k+2)}\leq k^{\frac{\theta-2}{4}}\leq\tK_2^{\frac{\theta-2}{4}}\overset{\eqref{def-tK3}}\leq\frac{1}{8L}\leq\frac{\sqrt{5}-1}{2L},
\]
and hence Lemma \ref{l-rec2} holds for such $k$.  Using Lemma \ref{l-rec2} with the choice of $\rho_k$ and $\eta_k$ given in \eqref{def-para3}, we obtain that
\beq\label{rec3}
h(x_{k+1})+2^{\theta-2}\gamma^2k^{\frac{\theta-2}{4}}[h(x_{k+1})]^{\theta}/\log(k+2)\leq h(x_k)+L_f^2k^{-\frac{\theta+2}{4}}/(2\log(k+2)).
\eeq
Further,  let
\beq\label{def-phi3}
\phi(t)=t+2^{\theta-2}\gamma^2k^{\frac{\theta-2}{4}}t^\theta/\log(k+2).
\eeq
Notice from \eqref{def-C3} that $C_2\geq1$. Using this and \eqref{def-phi3}, we have
\begin{align}
&\ \phi(C_2(k+1)^{-\frac12})-C_2k^{-\frac12}-L_f^2k^{-\frac{\theta+2}{4}}/(2\log(k+2))\nn\\
& \overset{\eqref{def-phi3}}{=}C_2^\theta 2^{\theta-2}\gamma^2k^{\frac{\theta-2}{4}}(k+1)^{-\frac{\theta}{2}}/\log(k+2)+C_2(k+1)^{-\frac12}-C_2k^{-\frac12}-L_f^2k^{-\frac{\theta+2}{4}}/(2\log(k+2))\nn\\
&\ \geq C_2^\theta 2^{\theta-2}\gamma^2k^{\frac{\theta-2}{4}}(k+1)^{-\frac{\theta}{2}}/\log(k+2)-C_2 k^{-\frac{3}{2}}/2-L_f^2k^{-\frac{\theta+2}{4}}/(2\log(k+2))\nn\\
&\ = \frac{k^{-\frac{\theta+2}{4}}}{\log(k+2)}\left(C_2^\theta 2^{\theta-2}\gamma^2\left(\frac{k}{k+1}\right)^{\frac{\theta}{2}}-C_2 k^{\frac{\theta-4}{4}}\log(k+2)/2-L_f^2/2\right)\nn\\
&\ \geq \frac{k^{-\frac{\theta+2}{4}}}{\log(k+2)}\left(C_22^{\frac{\theta}{2}-2}\gamma^2-C_2 k^{\frac{\theta-4}{4}}\log(k+2)/2-L_f^2/2\right).\label{l10-e1}
\end{align}
where the first inequality follows from $(k+1)^{-1/2}-k^{-1/2}\geq-k^{-3/2}/2$ thanks to the convexity of $t^{-1/2}$, and the second inequality is due to $\theta\geq1$, $C_2\geq1$ and $k/(k+1)\geq1/2$. In addition, one can verify that $t^{-1/2}\log(t+2)$ is decreasing on $[e^2,\infty)$. Using this, \eqref{def-tK3}, $1\leq\theta<2$ and $k\geq\tK_2\geq e^2$, we obtain that  
\[
k^{-\frac{1}{2}}\log(k+2)\leq\log(e^2+2)/e,\quad k^{\frac{\theta-2}{4}}
\leq\tK_2^{\frac{\theta-2}{4}}\overset{\eqref{def-tK3}}\leq e\gamma^2/(2^{2-\frac{\theta}{2}}\log(e^2+2)).
\]
 Multiplying both sides of these two inequalities yields $k^{\frac{\theta-4}{4}}\log(k+2)\leq2^{\frac{\theta}{2}-2}\gamma^2$, 
which together with \eqref{def-C3} implies that
\[
C_22^{\frac{\theta}{2}-2}\gamma^2-C_2 k^{\frac{\theta-4}{4}}\log(k+2)/2-L_f^2/2\geq C_22^{\frac{\theta}{2}-2}\gamma^2/2-L_f^2/2\overset{\eqref{def-C3}}{\geq}0.
\]
Using this, \eqref{rec3}, \eqref{def-phi3}, \eqref{l10-e1},  and the induction hypothesis that $h(x_k)\leq C_2k^{-1/2}$,  we obtain that
\[
\phi(C_2(k+1)^{-1/2})\geq C_2k^{-1/2}+\frac{L_f^2k^{-\frac{\theta+2}{4}}}{2\log(k+2)}\geq h(x_k)+\frac{L_f^2k^{-\frac{\theta+2}{4}}}{2\log(k+2)}\overset{\eqref{rec3}\eqref{def-phi3}}{\geq}\phi(h(x_{k+1})).
\]
It then follows from this inequality and the strict monotonicity of $\phi$ on $[0,\infty)$ that $h(x_{k+1})\leq C_2(k+1)^{-1/2}$. Hence, the induction is completed and the conclusion of this lemma holds. 
\end{proof}

\begin{lemma}\label{l-cnst4}
Let $\tK_3$, and $C_3$ be given in \eqref{def-tK4} and \eqref{def-nuC4}, respectively. 
Suppose that Assumption \ref{a1} holds, $\theta$ is given in Assumption \ref{a1}, and $\{x_k\}$ is generated by Algorithm \ref{alg2} with $\{\rho_k\}$, $\{\eta_k\}$ and $\{\alpha_k\}$ given in \eqref{def-para4}. Then  we have $\|c(x_k)\|^2\leq 2C_3k^{-1/\theta}$ for all $k\geq\tK_3$. 
\end{lemma}

\begin{proof}
The proof of this lemma follows from similar arguments as in the proof of Lemma \ref{l-cnst2} with $\nu=1/2$, and $\tK_1$ and $C_1$ replaced with $\tK_3$ and $C_3$, respectively. 
\end{proof}

The following lemma provides a relationship between $\bbE\left[\|g_{k+1}-\nabla f(x_{k+1})\|^2\right]$ and 
$\bbE\left[\|g_k-\nabla f(x_k)\|^2\right]$.

\begin{lemma}\label{l-exp2}
Suppose that Assumption \ref{a1} and \ref{a-mm} hold, and $\{g_k\}$ and $\{x_k\}$ are generated by Algorithm \ref{alg2}. Then for all $k\geq1$, we have
\[
\bbE\left[\|g_{k+1}-\nabla f(x_{k+1})\|^2\right]\leq(1-\alpha_k)\bbE\left[\|g_k-\nabla f(x_k)\|^2\right]+L_{\nabla f}^2\alpha_k^{-1}\bbE\left[\|x_{k+1}-x_k\|^2\right]+\sigma^2\alpha_k^2,
\]
where $\{\alpha_k\}$ is given in Algorithm \ref{alg2}, and $\sigma$ and $L_{\nabla f}$ are given in Assumptions \ref{a1} and \ref{a-storm}, respectively. 
\end{lemma}

\begin{proof}
Let $\Xi_k=\{\xi_1,\ldots, \xi_k\}$ denote the collection of samples drawn up to iteration $k-1$ in Algorithm 
\ref{alg2}.  It then follows from Assumption \ref{a1}(iii) that 
\[
\bbE[\nabla \tf(x_{k+1},\xi_{k+1})-\nabla f(x_{k+1})|\Xi_k]=0, \quad \bbE[\|\nabla \tf(x_{k+1},\xi_{k+1})-\nabla f(x_{k+1})\|^2|\Xi_k] \leq \sigma^2.
\]
Also, notice from \eqref{cf-bnd} that $\nabla f(x_{k+1})\in\cB(L_f)$ and hence $\nabla f(x_{k+1})=\Pi_{\cB(L_f)}(\nabla f(x_{k+1}))$. By these, the expression of $g_{k+1}$ in Algorithm \ref{alg2}, and the nonexpansiveness of the projection operator $\Pi_{\cB(L_f)}$, one has
\begin{align*}
&\bbE[\|g_{k+1}-\nabla f(x_{k+1})\|^2|\Xi_k]=\bbE[\|\Pi_{\cB(L_f)}\big((1-\alpha_k)g_k+\alpha_k\nabla \tf(x_{k+1},\xi_{k+1})\big)-\Pi_{\cB(L_f)}(\nabla f(x_{k+1}))\|^2|\Xi_k]\\ 
& \leq \bbE[\|(1-\alpha_k)g_k+\alpha_k\nabla \tf(x_{k+1},\xi_{k+1})-\nabla f(x_{k+1})\|^2|\Xi_k]\\ 
&=\bbE[\|(1-\alpha_k)(g_k-\nabla f(x_{k+1}))+\alpha_k(\nabla \tf(x_{k+1},\xi_{k+1})-\nabla f(x_{k+1}))\|^2|\Xi_k] \\
&=(1-\alpha_k)^2\|g_k-\nabla f(x_{k+1})\|^2+\alpha_k^2\bbE[\|\nabla \tf(x_{k+1},\xi_{k+1})-\nabla f(x_{k+1})\|^2|\Xi_k] \\
&\quad+2\alpha_k(1-\alpha_k)\langle g_k-\nabla f(x_{k+1}), \bbE[ \nabla \tf(x_{k+1},\xi_{k+1})-\nabla f(x_{k+1})|\Xi_k]\rangle  \\ 
& \leq (1-\alpha_k)^2\|g_k-\nabla f(x_{k+1})\|^2+\sigma^2 \alpha_k^2. 
\end{align*}
Taking expectation on both sides of this inequality yields
\beq\label{l7-e1}
\bbE[\|g_{k+1}-\nabla f(x_{k+1})\|^2] \leq (1-\alpha_k)^2\bbE[\|g_k-\nabla f(x_{k+1})\|^2]+\sigma^2 \alpha_k^2. 
\eeq
We divide the remainder of the proof by considering two separate cases: $\alpha_k=1$ and $0<\alpha_k<1$.

Case 1) $\alpha_k=1$. It follows from this and \eqref{l7-e1} that $\bbE[\|g_{k+1}-\nabla f(x_{k+1})\|^2] \leq \sigma^2 \alpha_k^2$ and hence the conclusion of this lemma clearly holds.

Case 2) $0<\alpha_k<1$.  By this, \eqref{l7-e1} and Assumption \ref{a-mm}, one has
\begin{align*}
&\bbE[\|g_{k+1}-\nabla f(x_{k+1})\|^2]\overset{\eqref{l7-e1}}\leq(1-\alpha_k)^2\bbE\left[\|g_k-\nabla f(x_k)+\nabla f(x_k)-\nabla f(x_{k+1})\|^2\right]+\sigma^2\alpha_k^2\\
&=(1-\alpha_k)^2\bbE\left[\|g_k-\nabla f(x_k)\|^2\right]+(1-\alpha_k)^2\bbE\left[\|\nabla f(x_k)-\nabla f(x_{k+1})\|^2\right]\\
&\quad+2(1-\alpha_k)^2\bbE\left[\langle g_k-\nabla f(x_k),\nabla f(x_k)-\nabla f(x_{k+1})\rangle\right]+\sigma^2\alpha_k^2\\
&\leq(1-\alpha_k)^2\bbE\left[\|g_k-\nabla f(x_k)\|^2\right]+(1-\alpha_k)^2\bbE\left[\|\nabla f(x_k)-\nabla f(x_{k+1})\|^2\right]\\
&\quad+(1-\alpha_k)^2\left(\frac{\alpha_k}{1-\alpha_k}\bbE\left[\|g_k-\nabla f(x_k)\|^2\right]+\frac{1-\alpha_k}{\alpha_k}\bbE\left[\|\nabla f(x_k)-\nabla f(x_{k+1})\|^2\right]\right)+\sigma^2\alpha_k^2\\
&=(1-\alpha_k)\bbE\left[\|g_k-\nabla f(x_k)\|^2\right]+(1-\alpha_k)^2\alpha_k^{-1}\bbE\left[\|\nabla f(x_k)-\nabla f(x_{k+1})\|^2\right]+\sigma^2\alpha_k^2\\
&\leq(1-\alpha_k)\bbE\left[\|g_k-\nabla f(x_k)\|^2\right]+L_{\nabla f}^2\alpha_k^{-1}\bbE\left[\|x_{k+1}-x_k\|^2\right]+\sigma^2\alpha_k^2,
\end{align*}
where the second inequality follows from $0<\alpha_k<1$ and Young's inequality, and the last inequality is due to Assumption \ref{a-mm}  and $0<\alpha_k<1$. Hence, the conclusion of this lemma also holds in this case.
\end{proof}

The next lemma provides an upper bound on $\bbE[Q_{\rho_k}(x_k)+\|g_k-\nabla f(x_k)\|^2]$.

\begin{lemma}\label{l-pfunc2}
Suppose that Assumptions \ref{a1} and \ref{a-mm} hold, and $\{g_k\}$ and $\{x_k\}$ are generated by Algorithm \ref{alg1} with $\eta_k \leq \alpha_k\leq 1$. 
Then for all $k\geq1$, we have
\begin{align*}
\bbE\left[Q_{\rho_k}(x_k)+\|g_k-\nabla f(x_k)\|^2\right]\leq&Q_{\rho_1}(x_1)+\|g_1-\nabla f(x_1)\|^2+\frac{1}{2}\sum_{i=1}^{k-1}(\rho_{i+1}-\rho_i)\bbE\left[\|c(x_{i+1})\|^2\right]\\
&+\frac{1}{2}\sum_{i=1}^{k-1}\left(L_{\nabla f}+\rho_iL-\eta_i^{-1}+2L_{\nabla f}^2\alpha_k^{-1}\right)\bbE\left[\|x_{i+1}-x_i\|^2\right]+\sigma^2\sum_{i=1}^{k-1}\alpha_i^2.
\end{align*}
where $\{\alpha_k\}$, $\{\rho_k\}$ and $\{\eta_k\}$ are given in Algorithm \ref{alg2},  $Q_\rho$  and $L$ are respectively defined in \eqref{def-Q} and \eqref{def-L}, and $\sigma$ and $L_{\nabla f}$  are given in Assumptions \ref{a1} and \ref{a-mm}, respectively. 
\end{lemma}

\begin{proof}
Observe from  \eqref{def-Q}, \eqref{def-L},  and Assumptions \ref{a1} and \ref{a-mm} that  $Q_{\rho_k}$ is 
$(L_{\nabla f}+\rho_k L)$-smooth. By this and similar arguments as for deriving \eqref{l5-e1},  one has that for all $k \geq 1$, 
\begin{align}
Q_{\rho_{k+1}}(x_{k+1})& \leq Q_{\rho_k}(x_k)+\frac{1}{2}\left(L_{\nabla f}+\rho_kL-\eta_k^{-1}\right)\|x_{k+1}-x_k\|^2+\frac{\eta_k}{2}\|g_k-\nabla f(x_k)\|^2 \nn \\
&\quad +\frac{1}{2}(\rho_{k+1}-\rho_k)\|c(x_{k+1})\|^2. \label{l11-ineq}
\end{align}
Notice from the assumption that $1-\alpha_k+\eta_k\leq1$. Using this, taking expectation on both sides of \eqref{l11-ineq}, and summing the resulting inequality with the inequality in Lemma \ref{l-exp2}, we obtain that
\begin{align}
&\bbE\left[Q_{\rho_{k+1}}(x_{k+1})+\|g_{k+1}-\nabla f(x_{k+1})\|^2\right]\nn\\
&\leq \bbE\left[Q_{\rho_k}(x_k)+\left(1-\alpha_k+\eta_k\right)\|g_k-\nabla f(x_k)\|^2\right]+\frac{1}{2}\left(L_{\nabla f}+\rho_kL-\eta_k^{-1}+2L_{\nabla f}^2\alpha_k^{-1}\right)\bbE\left[\|x_{k+1}-x_k\|^2\right]\nn\\
&\quad-\frac{\eta_k}{2}\bbE\left[\|g_k-\nabla f(x_k)\|^2\right]+\frac{1}{2}(\rho_{k+1}-\rho_k)\bbE\left[\|c(x_{k+1})\|^2\right]+\sigma^2\alpha_k^2\nn\\
&\leq \bbE\left[Q_{\rho_k}(x_k)+\|g_k-\nabla f(x_k)\|^2\right]+\frac{1}{2}\left(L_{\nabla f}+\rho_kL-\eta_k^{-1}+2L_{\nabla f}^2\alpha_k^{-1}\right)\bbE\left[\|x_{k+1}-x_k\|^2\right]\nn\\
&\quad-\frac{\eta_k}{2}\bbE\left[\|g_k-\nabla f(x_k)\|^2\right]+\frac{1}{2}(\rho_{k+1}-\rho_k)\bbE\left[\|c(x_{k+1})\|^2\right]+\sigma^2\alpha_k^2.\label{l11-e2}
\end{align}
The conclusion of this lemma follows by replacing $k$ with $i$ in the above inequalities and summing them up for all $1\leq i\leq k-1$.
\end{proof}

The following lemma provides an upper bound on $\dist^2\left(0,\nabla Q_{\rho_k}(x_{k+1})+\mcN_\mcX(x_{k+1})\right)$.

\begin{lemma}\label{l-dist2}
Suppose that Assumptions \ref{a1} and \ref{a-mm} hold, and $\{g_k\}$ and  $\{x_k\}$ are generated by Algorithm \ref{alg2}. Then for all $k\geq1$, we have
\[
\dist^2\left(0,\nabla Q_{\rho_k}(x_{k+1})+\mcN_\mcX(x_{k+1})\right)\leq3\left(\eta_k^{-2}+(L_{\nabla f}+\rho_kL)^2\right)\|x_{k+1}-x_k\|^2+3\|g_k-\nabla f(x_k)\|^2.
\]
where $\{\rho_k\}$ and $\{\eta_k\}$ are given in Algorithm \ref{alg2}, $L_{\nabla f}$ is given in Assumption \ref{a-mm}, and $L$ and $Q_\rho$ are defined in \eqref{def-L} and \eqref{def-Q}, respectively.
\end{lemma}

\begin{proof}
Recall from the proof of Lemma \ref{l-pfunc2} that $Q_{\rho_k}$ is 
$(L_{\nabla f}+\rho_k L)$-smooth. The proof of this lemma follows from this and similar arguments as in the proof of Lemma \ref{l-dist}.
\end{proof}

\begin{proof}[\textbf{Proof of Theorem \ref{thm:mm}}]
(i)  It follows from \eqref{def-para3}, \eqref{def-tK3} and the assumption $1\leq\theta<2$ that for all $i\geq\tK_2$, 
\begin{align}
&L_{\nabla f}+\rho_iL=L_{\nabla f}i^{-\frac{1}{2}}\eta_i^{-1}/\log(i+2)+Li^{\frac{\theta-2}{4}}\eta_i^{-1}/\log(i+2) \leq L_{\nabla f}\tK_2^{-\frac{1}{2}}\eta_i^{-1}+L\tK_2^{\frac{\theta-2}{4}}\eta_i^{-1}\leq\eta_i^{-1}/4,\label{thm2:ineq1} \\
&2L_{\nabla f}^2\alpha_i^{-1}=2 L_{\nabla f}^2\eta_i^{-1}/\log(i+2)\leq2 L_{\nabla f}^2\eta_i^{-1}/\log(\tK_2+2)\leq\eta_i^{-1}/4, \nn
\end{align}
which imply that
\beq \label{thm2:ineq2}
L_{\nabla f}+\rho_iL+2L_{\nabla f}^2\alpha_i^{-1}\leq\eta_i^{-1}/2\qquad\forall i\geq\tK_2.
\eeq
In addition, observe from \eqref{def-para3} that $\eta_k\leq\alpha_k\leq 1$ for all $k\geq1$. It then follows from the proof of Lemma \ref{l-pfunc2} that \eqref{l11-e2} holds.  By \eqref{def-Q}, \eqref{l11-e2}, \eqref{thm2:ineq1}, \eqref{thm2:ineq2},  Lemmas \ref{l-pfunc2} and \ref{l-dist2}, and similar arguments as for deriving \eqref{t1-sum2}, one can show that for all $k\geq2\tK_2$, 
\begin{align}
&\bbE\left[\dist^2\left(0,\nabla f(x_{\iota_k})+\rho_{\iota_k-1}\nabla c(x_{\iota_k})c(x_{\iota_k})+\mcN_\mcX(x_{\iota_k})\right)\right]\nn\\
&\leq \frac{51}{2(k-1)\eta_{k-1}}\Bigg( Q_1(x_1)-Q_1^*+\|g_1-\nabla f(x_1)\|^2+\frac{1}{2}\sum_{i=1}^{k-1}(\rho_{i+1}-\rho_i)\bbE\left[\|c(x_{i+1})\|^2\right]+\sigma^2\sum_{i=1}^{k-1}\alpha_i^2\nn\\
&\quad+\frac{1}{2}\sum_{i=1}^{\lceil k/2\rceil-1}\left(L_{\nabla f}+\rho_iL-\eta_i^{-1}+2L_{\nabla f}^2\alpha_i^{-1}\right)\bbE\left[\|x_{i+1}-x_i\|^2\right]\Bigg).\label{t2-sum1}
\end{align}
Further, one can observe that \eqref{t1-x1} also holds.  Recall that $1\leq\theta<2$ and $\|c(x_i)\| \leq C_c$ for all $i$. Using these, \eqref{def-para3}, \eqref{t2-ineq2},  \eqref{t1-x1}, and Lemma \ref{l-cnst3}, we have that for all $k\geq2\tK_2$,
\begin{align}
&\sum_{i=1}^{\tK_2-1}(\rho_{i+1}-\rho_i)\bbE\left[\|c(x_{i+1})\|^2\right]\leq C_c^2\sum_{i=1}^{\tK_2-1}\big((i+1)^{\frac{\theta}{4}}-i^{\frac{\theta}{4}}\big)=C_c^2\big(\tK_2^{\frac{\theta}{4}}-1\big),\label{sum-c1}\\
&\sum_{i=\tK_2}^{k-1}(\rho_{i+1}-\rho_i)\bbE\left[\|c(x_{i+1})\|^2\right]\leq 2C_2 \sum_{i=\tK_2}^{k-1}((i+1)^{\frac{\theta}{4}}-i^{\frac{\theta}{4}})(i+1)^{-\frac{1}{2}}\label{sum-c2}\\
&\leq\frac{1}{2}C_2\theta \sum_{i=\tK_2}^{k-1}i^{\frac{\theta-4}{4}}(i+1)^{-\frac{1}{2}}\leq\frac{1}{2}C_2\theta \sum_{i=\tK_2}^{k-1}i^{\frac{\theta-6}{4}}\leq\frac{C_2\theta(6-\theta) }{2(2-\theta)},\label{sum-c3}\\
&\sum_{i=1}^{k-1}\alpha_i^2\overset{\eqref{def-para3}}{=}\sum_{i=1}^{k-1}i^{-1}=1+ \sum_{i=2}^{k-1}i^{-1}\leq 1+\int^{k-1}_1 t^{-1}dt \leq 1+\log k,\label{sum-c4}\\
&\ \sum_{i=1}^{\lceil k/2\rceil-1}\left(\bL_{\nabla f}+\rho_iL-\eta_i^{-1}+2\alpha_i^{-1}\bL_{\nabla f}^2\right)\bbE\left[\|x_{i+1}-x_i\|^2\right]\nn\\
&\ =\sum_{i=1}^{\tK_2-1}\left(\bL_{\nabla f}+\rho_iL-\eta_i^{-1}+2\alpha_i^{-1}\bL_{\nabla f}^2\right)\bbE\left[\|x_{i+1}-x_i\|^2\right]+\sum_{i=\tK_2}^{\lceil k/2\rceil-1}\Big((\bL_{\nabla f}+\rho_iL-\eta_i^{-1}+2\alpha_i^{-1}\bL_{\nabla f}^2)\nn \\ 
&\ \quad\,\,\,\times \bbE\left[\|x_{i+1}-x_i\|^2\right]\Big) \nn \\
&\overset{\eqref{thm2:ineq2}}{\leq}\sum_{i=1}^{\tK_2-1}\left(\bL_{\nabla f}+\rho_iL-\eta_i^{-1}+2\alpha_i^{-1}\bL_{\nabla f}^2\right)\bbE\left[\|x_{i+1}-x_i\|^2\right]\nn\\
&\overset{\eqref{t1-x1}}{\leq}2\sum_{i=1}^{\tK_2-1}\left(\bL_{\nabla f}+\rho_iL+2\alpha_i^{-1}\bL_{\nabla f}^2\right)\eta_i^2\left(L_f^2+C_c^2L_c^2\rho_i^2\right)\nn\\
&\overset{\eqref{def-para3}}{\leq}2(1+\log\tK_2)\big(\bL_{\nabla f}+\tK_2^{\frac{\theta}{4}}L+2\bL_{\nabla f}^2\tK_2^{\frac12}\big)\big(L_f^2+C_c^2L_c^2\tK_2^{\frac{\theta}{2}}\big),\label{sum-c5}
\end{align}
where the inequality in \eqref{sum-c1} follows from \eqref{def-para3} and $\|c(x_i)\| \leq C_c$ for all $i$, \eqref{sum-c2} is due to \eqref{def-para3} and Lemma \ref{l-cnst3}, the first inequality in \eqref{sum-c3} follows from $(i+1)^{\theta/4}-i^{\theta/4}\leq\theta i^{(\theta-4)/4}/4$ for all $i\geq1$ thanks to the concavity of $t^{\theta/4}$ with $1\leq\theta<2$,  the third inequality in \eqref{sum-c3} is due to 
\[
\sum_{i=\tK_2}^{k-1}i^{(\theta-6)/4}\leq \sum_{i=1}^{k-1}i^{(\theta-6)/4} = 1+\sum_{i=2}^{k-1}i^{(\theta-6)/4} \leq 1+\int^{\infty}_1 t^{(\theta-6)/4} dt = 1+\frac{4}{2-\theta},
\]
and the last inequality in \eqref{sum-c5} follows from the relations $\rho_i \leq \tK_2^{\theta/4}$ for $1\leq i\leq \tK_2-1$ and
$\sum_{i=1}^{\tK_2-1}\eta_i^2=\sum_{i=1}^{\tK_2-1}i^{-1}\leq 1+\log\tK_2$ due to the  choice of 
$\rho_i$ and $\eta_i$ in \eqref{def-para3}. 

Using \eqref{def-para3}, \eqref{t2-sum1}, \eqref{sum-c1}, \eqref{sum-c2}, \eqref{sum-c4} and \eqref{sum-c5}, we have
\begin{align*}
&\bbE\left[\dist^2\left(0,\nabla f(x_{\iota_k})+\rho_{\iota_k-1}\nabla c(x_{\iota_k})c(x_{\iota_k})+\mcN_\mcX(x_{\iota_k})\right)\right]\\
&\leq\frac{51\log(k+2)}{2(k-1)^{\frac{1}{2}}}\Bigg( Q_1(x_1)-Q_1^*+\|g_1-\nabla f(x_1)\|^2+\frac{C_2\theta(6-\theta) }{4(2-\theta)}+\frac{1}{2}C_c^2\big(\tK_2^{\frac{\theta}{4}}-1\big)\nn\\
&\quad+\sigma^2(1+\log k)+(1+\log\tK_2)\big(L_{\nabla f}+L\tK_2^{\frac{\theta}{4}}+2L_{\nabla f}^2\tK_2^{\frac{1}{2}}\big)\big(L_f^2+C_c^2L_c^2\tK_2^{\frac{\theta}{2}}\big)\Bigg)\qquad\qquad\forall k\geq2\tK_2.
\end{align*}
By this, \eqref{def-Q}, $\iota_k>\lceil k/2\rceil\geq\tK_2$ for all $k\geq2\tK_2$, and Lemma \ref{l-cnst3} with $k$ replaced by $\iota_k$, one can see that statement (i) of Theorem \ref{thm:mm} holds.

(ii) It follows from \eqref{def-para4} and \eqref{def-tK4} that for all $i\geq\tK_3$,
\begin{align*}
L_{\nabla f}+\rho_iL &=L_{\nabla f}i^{-\frac{1}{2}}\eta_i^{-1}/(4\log(i+2))+L\eta_i^{-1}/(4\log(i+2))\\
&\leq L_{\nabla f}\tK_3^{-\frac{1}{2}}\eta_i^{-1}/4+L\eta_i^{-1}/(4\log(\tK_3+2))\leq\eta_i^{-1}/4,\\
2\alpha_i^{-1}L_{\nabla f}^2& =2L_{\nabla f}^2\eta_i^{-1}/(4\log(i+2))\leq L_{\nabla f}^2\eta_i^{-1}/(2\log(\tK_3+2))\leq\eta_i^{-1}/4,
\end{align*}
which imply that 
\beq\label{t2-ineq2}
L_{\nabla f}+\rho_iL+2L_{\nabla f}^2\alpha_i^{-1}\leq\eta_i^{-1}/2\qquad\forall i\geq\tK_3.
\eeq
By this, Lemma \ref{l-cnst3}, \eqref{def-para4}, \eqref{t1-x1}, $\|c(x_i)\| \leq C_c$ for all $i$, and similar arguments as for deriving \eqref{sum2-c1},  \eqref{sum2-c2} and \eqref{sum2-c4},  
one can show that for all $k\geq2\tK_3$, 
\begin{align*}
&\sum_{i=\tK_3}^{k-1}(\rho_{i+1}-\rho_i)\bbE\left[\|c(x_{i+1})\|^2\right]\leq C_3\max\{1,k^{\frac12-\frac{1}{\theta}}\}(1+\log k),\nn\\
&\sum_{i=1}^{\tK_3-1}(\rho_{i+1}-\rho_i)\bbE\left[\|c(x_{i+1})\|^2\right]\leq C_c^2\big(\tK_3^{\frac{1}{2}}-1\big),\\
&\sum_{i=1}^{\lceil k/2\rceil-1}\left(L_{\nabla f}+\rho_iL-\eta_i^{-1}+2L_{\nabla f}^2\alpha_i^{-1}\right)\bbE\left[\|x_{i+1}-x_i\|^2\right]\nn\\
&\leq
\frac{1}{8}(1+\log\tK_3)\big(L_{\nabla f}+L\tK_3^{\frac{1}{2}}+2L_{\nabla f}^2\tK_3^{\frac{1}{2}}\big)\big(L_f^2+C_c^2L_c^2\tK_3\big).
\end{align*}
In addition, observe from \eqref{def-para4} that $\eta_k\leq\alpha_k\leq 1$ for all $k\geq1$. Using this,  \eqref{t2-ineq2}, and similar arguments as in the proof of statement (i) of this theorem,  we can see that \eqref{t2-sum1} holds for all $k\geq2\tK_3$. Also, by \eqref{def-para4} and \eqref{sum-c4}, one has 
$\sum_{i=1}^{k-1}\alpha_i^2 \leq 1+\log k$ for all $k\geq1$.  Using this, \eqref{def-para4}, \eqref{t2-sum1}, and the above three inequalities,  we have
\begin{align*}
&\bbE\left[\dist^2\left(0,\nabla f(x_{\iota_k})+\rho_{\iota_k-1}\nabla c(x_{\iota_k})c(x_{\iota_k})+\mcN_\mcX(x_{\iota_k})\right)\right]\\
&\leq \frac{102\log(k+2)}{(k-1)^{\frac{1}{2}}}\Bigg( Q_1(x_1)-Q_1^*+\|g_1-\nabla f(x_1)\|^2+\frac{1}{2} C_3\max\{1,k^{\frac{1}{2}-\frac{1}{\theta}}\}(1+\log k)+\frac{1}{2}C_c^2\big(\tK_3^{\frac{1}{2}}-1\big)\nn\\
&\quad+\sigma^2(1+\log k)+\frac{1}{16}(1+\log\tK_3)\big(L_{\nabla f}+L\tK_3^{\frac{1}{2}}+2L_{\nabla f}^2\tK_3^{\frac{1}{2}}\big)\big(L_f^2+C_c^2L_c^2\tK_3\big)\Bigg)\qquad\qquad \forall k\geq2\tK_3.
\end{align*}
By this, \eqref{def-Q}, $\iota_k>\lceil k/2\rceil\geq\tK_3$ for all $k\geq2\tK_3$, and Lemma \ref{l-cnst4} with $k$ replaced by $\iota_k$, one can see that statement (ii) of Theorem \ref{thm:mm} holds.
\end{proof}

\section{Concluding remarks}\label{sec:conclude}

In this paper, we study a class of deterministically constrained stochastic optimization problems and propose single-loop variance-reduced stochastic first-order methods with complexity guarantees for finding an $\epsilon$-\emph{surely feasible} stochastic stationary point ($\epsilon$-SFSSP), which is stronger than those targeted by existing methods---specifically, one in which the constraint violation is within $\epsilon$ with \emph{certainty}, and the expected first-order stationarity violation is within $\epsilon$.

Although we focus on stochastic optimization with deterministic equality constraints only, the proposed methods and complexity results can be directly extended to the following problem:
\[
\min_{x\in\mcX} \big\{\bbE[\tilde f(x,\xi)]:  c_\mcE(x)=0, \ c_\mcI(x) \leq 0\big\},
\] 
where $c_\mcE$ and $c_\mcI$ are smooth mappings, and $\tilde f$ and $\mcX$ are as defined in Section \ref{intro}. Specifically, to solve this problem, Algorithms \ref{alg1} and \ref{alg2} can be modified by updating $G_k$ as
\[
G_k=g_k+\rho_k\big(\nabla c_\mcE(x_k)c_\mcE(x_k)+ \nabla c_\mcI(x_k)[c_\mcI(x_k)]_+\big).
\]
To establish the complexity of the resulting algorithms, we can replace \eqref{err-cond} with
\[
\dist\big(0,\nabla c_\mcE(x)c_\mcE(x)+\nabla c_\mcI(x)[c_\mcI(x)]_++\mcN_\mcX(x)\big)\geq\gamma\|(c_\mcE(x), [c_\mcI(x)]_+)\|^\theta \qquad \forall x\in X,
\]
$Q_\rho$ with $Q_\rho(x) = f(x) + \rho (\|c_\mcE(x)\|^2 + \|[c_\mcI(x)]_+\|^2)/2$, and $h$ with
\[
h(x)=(\|c_\mcE(x)\|^2+\|[c_\mcI(x)]_+\|^2)/2.
\]

For future work, it would be interesting to investigate whether an $\epsilon$-SFSSP satisfying \eqref{eps-stationarity2} can be obtained via a sample average approximation approach (e.g., see \cite{shapiro2021lectures}) with sample and first-order operation complexities comparable to those of our proposed methods. Additionally, we plan to conduct computational studies to evaluate and compare the performance of our methods against existing approaches.

\end{document}